\documentclass[11pt]{amsart}
\usepackage{cite}
\usepackage{wrapfig, lipsum,booktabs}
\usepackage{graphicx}
\usepackage{amssymb}
\usepackage{epstopdf}
\usepackage{verbatim}
\usepackage{bm}
\usepackage{multicol}
\usepackage{multirow}
\usepackage{subfigure}
\usepackage{fancyhdr} 
\usepackage{float}
\usepackage{stmaryrd}
\usepackage{color}
\usepackage{soul}
\usepackage{tikz}
\usepackage{todonotes}
\usetikzlibrary{patterns}
\usepackage{pgfplots}
\usepackage{cancel}
\newtheorem{theorem}{Theorem}[section]
\newtheorem{lemma}{Lemma}[section]
\newtheorem{remark}{Remark}[section]

\usepackage{multirow}
\usepackage{amsmath,amssymb,eucal}
\usepackage{graphicx,subfigure,epsfig}
\usepackage{psfrag}
\usepackage{url}
\usepackage[top=1in, bottom=1.in, left=1in, right=1in]{geometry}

\newcommand{\bld}[1]{\hbox{\boldmath$#1$}}    
\newcommand{\Th}{\mathcal{T}_h}

\newcommand{\Eh}{\mathcal{E}_h}

\newcommand{\lXi}{\text{\Large$\Xi$}}
\setulcolor{red}
\newcommand{\jmp}[1]{[\![#1 ]\!]}

\begin{document}
\title[Div-HDG]{A monolithic divergence-conforming HDG scheme
for a linear fluid-structure interaction model}
\author{Guosheng Fu}
\address{Department of Applied and Computational Mathematics and 
Statistics, University of Notre Dame, USA.}
\email{gfu@nd.edu}
 \thanks{G. Fu gratefully acknowledges the partial support of this work
 from U.S. National Science Foundation through grant DMS-2012031.}
\author{Wenzheng Kuang}
\address{Department of Applied and Computational Mathematics and 
Statistics, University of Notre Dame, USA.}
\email{wkuang1@nd.edu}

\begin{abstract}
We present a novel monolithic divergence-conforming HDG scheme for 
a linear fluid-structure interaction (FSI) problem
with a thick structure. 
A pressure-robust optimal energy-norm estimate 
is obtained for the semidiscrete scheme.
When combined with a Crank-Nicolson time discretization,
our fully discrete scheme is energy stable and produces an exactly divergence-free fluid velocity approximation.
The resulting linear system, which is symmetric and
indefinite, is solved using 
a preconditioned MinRes method  with a robust block algebraic multigrid (AMG) preconditioner. 
\end{abstract}
\maketitle

\section{Introduction}
\label{sec:intro}
Fluid-structure interaction (FSI) describes a multi-physics phenomenon that involves the highly non-linear coupling between a deformable or moving structure and a surrounding or internal fluid. There has been intensive interest in solving FSI problems due to its wide applications in biomedical, engineering and architecture fields, such as the simulation of blood-cell interactions, the study of wing fluttering in aerodynamics and the design of dams with reservoirs. However, it is generally difficult to achieve analytical solution to FSI problem with its nonlinear and multi-physics nature. Instead, there have been extensive studies in its numerical solutions and an increasing demand for more efficient and accurate numerical schemes 
\cite{bungartz2006fluid,chakrabarti2005numerical,dowell2001modeling,hou2012numerical,richter2017fluid}.

Numerical methodologies for solving FSI problems can be roughly categorized into partitioned and monolithic schemes. 
Distinct mechanisms in fluid and structure domains naturally suggest solvers using partitioned schemes \cite{farhat2006provably,NobileVergara08}. This numerical procedure treats each physical phenomenon separately and allows the use of existing software frameworks that are well-established for each subproblem. However, the design of efficient partitioned schemes that produce stable and accurate results remains a challenge, especially when the density of fluid is comparable to that of structure due to numerical instabilities known as added mass effect \cite{causin2005added}. The design and analysis of partitioned schemes to circumvent such problems has been an active research area in the past decade \cite{causin2005added,fernandez2013fully,lukavcova2013kinematic,banks2016added,bukac2020refactorization}.
An alternative to partitioned strategy is the monolithic approach, which solves the fluid flow and structure dynamics simultaneously using one unified fully-coupled formulation \cite{hubner2004monolithic,tezduyar2007modelling,rugonyi2001finite}. The boundary conditions on the fluid-structure interface will be automatically satisfied in the procedure. Monolithic schemes are usually more robust than partitioned schemes and allow more rigorous analysis of discretization and solution techniques \cite{kloppel2011fluid,richter2017fluid}. However, monolithic schemes have been criticized for requiring well-designed preconditioners\cite{gee2011truly,Nobile01,badia2008modular}, more memory and computation time since the whole system is solved in one formulation.

In this paper, we present a novel monolithic divergence-conforming HDG scheme for 
a linear FSI problem
with a thick structure. 
The fluid Stokes problem is discretized using the divergence-free HDG
scheme of Lehrenfeld and Sch\"oberl \cite{Lehrenfeld10,
LehrenfeldSchoberl16}, and the structure linear elasticity problem is discretized using
the divergence-conforming HDG scheme of Fu et al. \cite{FuLehrenfeld20}. 
We approximate the fluid and structure velocities together using a single 
$H(\mathrm{div})$-conforming finite element space, and we also introduce
a global (hybrid) unknown that approximate the tangential component of the
velocities on the mesh skeleton for the purpose of efficient
implementation.
A pressure-robust optimal energy-norm estimate 
is obtained for the resulting semidiscrete scheme.
We then use a Crank-Nicolson time discretization, and the fluid-structure interface conditions are naturally treated
monolithically. Our fully discrete scheme produces an exactly
divergence-free fluid velocity approximation and is energy stable.

When polynomials of degree $k\ge 1$ is used in the scheme, 
the global linear system, which is symmetric and indefinite,
consists of degrees of freedom (DOFs) for 
the normal component of velocity (of polynomial degree $k$) on the mesh
skeleton (facets), the tangential hybrid velocity (of polynomial degree $k-1$)
on the mesh skeleton, and one pressure DOF per element on the
mesh. The linear system problem is then solved via a preconditioned MinRes
method \cite{Saad03} with a block diagonal preconditioner which is of similar form as the uniform
preconditioner studied in Olshanskii  et al. \cite{Olshanskii06}
for a generalized Stokes interface problem. We further use an auxiliary
space preconditioner of Xu \cite{Xu96}
with algebraic multigrid (AMG) for the velocity block and a hypre AMG preconditioner
for the pressure block to arrive at the final block AMG preconditioner.
This preconditioner is numerically verified to be robust with respect to
mesh size, time step size, and material parameters.


The rest of the paper is organized as follows.
In Section \ref{sec:hdg}, we introduce the spatial and temporal
discretization of the divergence-conforming HDG scheme for 
a linear FSI problem with a thick structure.
We then present the block AMG preconditioner in 
Section \ref{sec:prec}.
The a priori error analysis of the semidiscrete scheme is performed in 
Section \ref{sec:err}.
Numerical results are presented in Section \ref{sec:num}.
We conclude in Section \ref{sec:conclude}.

\section{The monolithic divergence-conforming 
HDG scheme for a linear FSI model}
\label{sec:hdg}
\subsection{The model FSI problem}
We consider the interaction between an incompressible, viscous fluid and
a elastic structure. We denote by $\Omega^f(t)\subset \mathbb{R}^d$ the domain occupied by the
fluid and $\Omega^s(t)\subset \mathbb{R}^d$, $d=2,3$ by the solid at the time 
$t\in[0,T]$. Let $\Gamma(t) = \overline{\Omega^f} 
\cap \overline{\Omega^s}$ be the part of the boundary where the elastic
solid interacts with the fluid; see Fig. \ref{fig:fsi}.
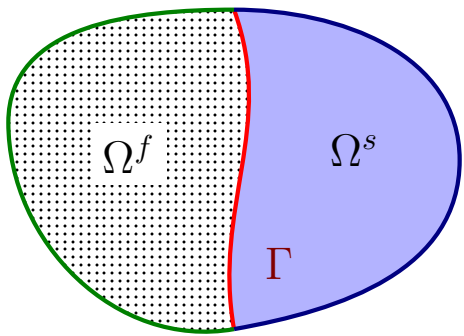
\begin{figure}[ht!]
  \begin{center}
  \begin{tikzpicture}
    \draw[thick,pattern=dots,draw=none] 
    (5,3.75)
    to[out=190,in=-90] (2,6.5) 
    to[out=90,in=180] (5,8)
    to[out=-70,in=100] (5,3.75) --cycle;
    \draw[thick,fill=blue!30,draw=none] 
    (5,3.75)
    to[out=10,in=-90] (8,6) 
    to[out=90,in=0] (5,8)
    to[out=-70,in=100] (5,3.75) --cycle;
    \draw[ultra thick,red] (5,8)
    to[out=-70,in=100] (5,3.75);
    \draw[ultra thick,draw=blue!50!black] 
    (5,3.75)
    to[out=10,in=-90] (8,6) 
    to[out=90,in=0] (5,8);
    \draw[ultra thick,draw=green!50!black] 
    (5,3.75)
    to[out=190,in=-90] (2,6.5) 
    to[out=90,in=180] (5,8);
    \node at (3.6,6.5)[fill=white,below,scale=1.2] {\color{black} \Large
    $\Omega^f$};
    \node at (6.6,6.5)[below,scale=1.2] {\color{black} \Large $\Omega^s$};
    \node at (5.6,5)[below,scale=1.2] {\color{red!50!black} \Large
    $\Gamma$};
  \end{tikzpicture}
\end{center}
  \caption{\it Sketch of a domain for FSI.}
  \label{fig:fsi}
\end{figure}
For the purpose of this paper, we assume that the
nonlinear convection term in the fluid is negligible and
the solid is linearly elastic and the deformation is small. 
Hence, the domain $\Omega^{f/s}$ does not change over time, and 
the fluid flow is modeled using the time dependent Stokes equations
while the structure is modeled using the linear elastodynamics equations: 
\begin{subequations}
  \label{model}
  \begin{alignat}{2}
    \label{f-eq}
\left.
  \begin{tabular}{rl}
$\rho^f \partial_t \bld u^f -\nabla \cdot \bld \sigma^f(\bld u^f, p^f) =$&
    $\bld f^f$\quad\\
$\nabla\cdot \bld u^f =$ &$0$
  \end{tabular}\right\} &\;\; \text{in
} \Omega^f\times [0,T],\\
    \label{s-eq}
\left.
  \begin{tabular}{rl}
$\rho^s \partial_t \bld u^s -\nabla \cdot \bld \sigma^s(\bld \eta^s) =$&
    $\bld f^s$\quad\\
$\partial_t\bld \eta^s-\bld u^s =$ &$0$
  \end{tabular}\right\} &\;\; \text{in
}     \Omega^s\times [0,T],
  \end{alignat}
where  $\rho^f$ is the fluid density, 
 $\bld u^f$ is the fluid velocity, $p^f$ is the fluid pressure,
$\bld f^f$ is the fluid source term, and
 $\bld \sigma^f$ is the fluid stress
 tensor given as follows:
  \begin{align*}
    \bld \sigma^f(\bld u^f, p^f): =  
    -p^f\bld I + 2\mu^f\bld D(\bld u^f),
  \end{align*}
  where $\bld I$ is the identity tensor, 
  $\mu^f$ is the fluid viscosity, and 
  $\bld D(\bld u^f):=\frac12(\nabla\bld u^f+(\nabla \bld u^f)^T)$
  is the fluid strain rate tensor, 
  while  $\rho^s$ is the structure density, 
  $\bld \eta^s$ is the structure displacement,
 $\bld u^s$ is the structure velocity, 
 $\bld f^s$ is the structure source term, and
 $\bld \sigma^s$ is the structure Cauchy stress
 tensor given as follows:
  \begin{align*}
    \bld \sigma^s(\bld \eta^s): =  
    \lambda^s(\nabla \cdot\bld\eta^s)\bld I + 2\mu^s\bld D(\bld \eta^s),
  \end{align*}
where $\mu^s$ and $\lambda^s$ are the Lam\'e constants.

The fluid and structure sub-problems are coupled with the following {\it kinematic}
and 
{\it dynamic} coupling conditions \cite{richter2017fluid} on the interface $\Gamma$:
\begin{alignat}{2}
  \label{interface}
  \left.
  \begin{tabular}{r}
 $ \bld u^f = \; \bld u^s$\\
 $\bld \sigma^f\bld n^f+\bld \sigma^s\bld n^s =\;0$
 \end{tabular}\right\} \text{ on }\Gamma\times [0,T],
\end{alignat}
where 
$\bld n^f$ and $\bld n^s$ are the normal
directions on the fluid-structure interface $\Gamma$ 
pointing from the fluid and structure domains,
respectively.

To close the system, we need proper initial and boundary conditions. For
simplicity, in our analysis we consider a homogeneous Dirichlet boundary conditions 
on the exterior boundaries:
\begin{align}
  \label{bcbc}
  \bld u^f =&\;  0 \;\;\text{ on
  } \Gamma^f:=\partial\Omega^f\backslash\Gamma,\quad\quad
  \bld \eta^s =\;0 \;\;\text{ on } \Gamma^s:=\partial\Omega^s\backslash\Gamma.
\end{align}
We mention that other standard boundary conditions on the exterior boundaries 
can also be used,  see e.g. the numerical results in Section \ref{sec:num}.
Finally, the initial condition is given as follows:
\begin{align}
  \label{init}
  \bld u^f(x, 0) =&\;  \bld u^f_0(x) \;\;\text{ on
  } \Omega^f,\quad\quad
  \bld u^s(x,0) =\;\bld u^s_0(x), \,
  \bld \eta^s(x,0) =\;\bld \eta^s_0(x), \,
  \;\;\text{ on } \Omega^s,
\end{align}
\end{subequations}
where $\bld u^f_0, \bld u^s_0$, and $\bld \eta^s_0$, respectively, are the 
initial fluid velocity, initial structure velocity, and initial structure
displacement, respectively.

\subsection{Preliminaries and finite element spaces}
We assume the domains $\Omega^f$, $\Omega^s$, as well as the interface 
$\Gamma$ are polypope.
Let $\Omega$ be the union of the fluid and structure domains, i.e., 
$\overline{\Omega} = \overline{\Omega^f}\cup \overline{\Omega^s}$.
Let $\Th$ be an interface-fitted conforming simplicial triangulation of the domain
$\Omega$ such that the interface $\Gamma$ is the union of element facets.
For any element $K\in \Th$, we denote by $h_K$ its diameter and we denote
by 
$h$ the maximum diameter over all mesh elements.
Denote by $\Th^f$ the set of mesh elements that belong to $\Omega^f$ and by 
$\Th^s$ those belong to $\Omega^s$. Denote by $\Eh$ the set of facets of
$\Th$, by $\Eh^f$ the set of facets
that are interior to $\overline{\Omega^f}$, and by 
$\Eh^s$ the set of facets that are interior to $\overline{\Omega^s}$.
We also denote by $\Gamma_h$, $\Gamma_h^f$, $\Gamma_h^s$ the set of facets that lie on the interface
$\Gamma$, the fluid exterior boundary $\Gamma^f$, and the 
solid exterior boundary $\Gamma^s$, respectively. We have $\Gamma_h = \Eh^f\cap \Eh^s$.
Given a simplex $S\subset \mathbb{R}^d, d=1,2,3$, we denote 
$\mathcal{P}^m(S)$, $m\ge 0$, as the space of polynomials of degree at most
$m$. 
Given a facet $F\in \Eh$ with normal direction $\bld n$, we denote 
$\mathsf{tang}(\bld w):=\bld w-(\bld w\cdot\bld n)\bld n$ as the {\it
tangential component} of a vector field $\bld w$. 

The following finite element spaces will be used in our scheme:
\begin{subequations}
  \label{spaces}
\begin{align}
    \bld V_h^{r}:=&\;\{\bld v\in H(\mathrm{div};\Omega): \;\;\bld v|_K\in
  [\mathcal{P}^r(K)]^d, \;\forall K\in \Th\}, \\ 
  \label{vh0}
      \bld V_{h,0}^{r}:=&\;\{\bld v\in \bld V_h^{r}: \;\;\bld v\cdot
      \bld n|_F = 0,\; \forall F\in \Gamma_h^f\cup\Gamma_h^s\}, \\
        \widehat{\bld V}_h^{r}:=&\;\{\widehat{\bld v}
          \in [L^2(\Eh)]^d: \;\;\widehat{\bld v}|_F\in
        [\mathcal{P}^r(F)]^d,\;\;\widehat{\bld v}\cdot \bld n|_F=0, 
      \;\forall F\in \Eh\}, \\ 
  \label{vhath0}
          \widehat{\bld V}_{h,0}^{r}:=&\;\{\widehat{\bld v}\in 
            \widehat{\bld V}_h^{r,s}: \;\;\mathsf{tang}(\widehat{\bld
          v})|_F = 0,\; \forall F\in \Gamma_h^f\cup \Gamma_h^s\}, \\
  Q_h^{r}:=&\;\{q\in L^2(\Omega): \;\;q|_K\in \mathcal{P}^r(K),
  \;\forall K\in \Th\},
\end{align}
where $r\ge 0$ is the polynomial degree.
\end{subequations}
We further use a superscript $f/s$ to indicate the restriction of these
spaces on the fluid/structure domain, that is,  
\begin{align*}
  \bld V_{h,0}^{r,f}:=\{\bld v|_{\Th^f}:\;\bld v\in\bld V_{h,0}^{r}\}, \;\;
  \bld V_{h,0}^{r,s}:=&\{\bld v|_{\Th^s}:\;\bld v\in\bld V_{h,0}^{r}\},\\
  \widehat{\bld V}_{h,0}^{r,f}:=\{\widehat{\bld v}|_{\Eh^f}:\;\widehat{\bld
v}\in\widehat{\bld V}_{h,0}^{r}\}, \;\;
  \widehat{\bld V}_{h,0}^{r,s}:=&\{\widehat{\bld v}|_{\Eh^s}:\;\widehat{\bld v}\in
  \widehat{\bld V}_{h,0}^{r}\},\\
Q_h^{r,f}:=\{q|_{\Th^f}:\;q\in Q_h^{r}\}, \;\;
Q_h^{r,s}:=&\{q|_{\Th^s}:\;q\in Q_h^{r}\}.
\end{align*}

\subsection{Semi-discrete divergence-conforming HDG scheme}
In this subsection, we present the divergence-conforming HDG spatial
discretization 
  \cite{Lehrenfeld10,
LehrenfeldSchoberl16, FuLehrenfeld20} 
of the linear FSI system \eqref{model}.

We use the globally divergence-conforming finite element space 
$\bld V_h^{r}$ in
\eqref{vh0} to approximate the global velocity 
\begin{align}
  \label{vel}
  \bld u=\left\{\begin{tabular}{ll}
    $\bld u^f$ & on $\Omega^f$,\\
    $\bld u^s$ & on $\Omega^s$
\end{tabular}\right.,
\end{align}
 and 
the global tangential facet finite element space $\widehat{\bld V}_h^{r}$
in \eqref{vhath0} to approximate the tangential component of 
the global velocity $\bld u$ on the mesh skeleton. 

The weak formulation of the divergence-conforming HDG scheme with 
polynomial degree $k\ge 1$ for 
\eqref{model} is given as follows: 
Find $(\bld u_h, \widehat{\bld u}_h, p_h^f, \bld \eta_h^s, \widehat{\bld
\eta}_h^s)\in \bld V_{h,0}^k\times \widehat{\bld V}_{h,0}^{k-1}\times 
Q_h^{k-1,f}\times\bld V_{h,0}^{k,s}\times \widehat{\bld V}_{h,0}^{k-1,s} 
$ such that
\begin{subequations}
  \label{semi}
  \begin{alignat}{2}
  \label{semi-1}
    (\rho \partial_t\bld u_h, \bld v_h)+ 
    2\mu^f A_h^f(
    (\bld u_h, \widehat{\bld u}_h),(\bld v_h, \widehat{\bld v}_h) )
    -(p_h^f, \nabla\cdot \bld v_h)_f
    -(\nabla\cdot\bld u_h, q_h^f)_f\;\;
    &\\
   + 2\mu^s A_h^s(
 (\bld \eta^s_h, \widehat{\bld \eta}^s_h),(\bld v_h, \widehat{\bld v}_h) )
    +\lambda^s(\nabla \cdot \bld \eta_h^s, \nabla \cdot\bld v_h)_s
    &\;
    =\;  (\bld f, \bld v_h), \nonumber\\
  \label{semi-2}
    (\partial_t \bld \eta_h^s-\bld u_h, \bld \xi_h^s)_s &\;= \;0,\\
  \label{semi-3}
    \langle\partial_t \widehat{\bld \eta}_h^s-\widehat{\bld u}_h, 
    \widehat{\bld \xi}^s_h\rangle_s&\;=\;0,
  \end{alignat}
  for all 
$(\bld v_h, \widehat{\bld v}_h, q_h^f, \bld \xi_h^s, \widehat{\bld
\xi}_h^s)\in \bld V_{h,0}^k\times \widehat{\bld V}_{h,0}^{k-1}\times 
Q_h^{k-1,f}\times\bld V_{h,0}^{k,s}\times \widehat{\bld V}_{h,0}^{k-1,s} 
$, where 
$(\cdot, \cdot)$ denotes the $L^2$-inner product on the domain $\Omega$, 
$(\cdot, \cdot)_f$ denotes the $L^2$-inner product on the fluid domain
$\Omega^f$, 
$(\cdot, \cdot)_s$ denotes the $L^2$-inner product on the structure domain
$\Omega^s$, and 
$\langle\cdot,\cdot\rangle_s$ denotes the $L^2(\Eh^s)$-inner product on the
structure mesh skeleton $\Eh^s$, 
moreover, $\rho= \left\{\begin{tabular}{ll}
    $\rho^f$ & on $\Omega^f$\\
    $\rho^s$ & on $\Omega^s$
\end{tabular}\right.
$ is the global density and 
$\bld f= \left\{\begin{tabular}{ll}
    $\bld f^f$ & on $\Omega^f$\\
    $\bld f^s$ & on $\Omega^s$
\end{tabular}\right.
$ is the global source term on $\Omega$.
Here the operators $A_h^f$ and $A_h^s$ are the following
symmetric interior penalty HDG diffusion operators with 
a {\it projected
jumps} formulation:
for $i\in\{f, s\}$, 
\end{subequations}
\begin{align}
  \label{hdg-diff}
  A_h^i((\bld v_h, \widehat{\bld v}_h), (\bld w_h, \widehat{\bld
  w}_h)):=\sum_{K\in\Th^i}&\;\int_{K}\bld D(\bld v_h):
  \bld D(\bld w_h)\,\mathrm{dx}
  -\int_{\partial K}
  \bld D(\bld v_h)\bld n\cdot \mathsf{tang}(\bld w_h-\widehat{\bld w}_h)
  \,
  \mathrm{ds}\\       
&\hspace{-26ex}
  -\int_{\partial K}
  \bld D(\bld w_h)\bld n\cdot \mathsf{tang}(\bld v_h-\widehat{\bld v}_h)
  \,
  \mathrm{ds}
  +
  \int_{\partial K}
  \frac{\alpha k^2}{h}
  \Pi_h(\mathsf{tang}(\bld v_h-\widehat{\bld v}_h))
  \cdot \Pi_h(\mathsf{tang}(\bld w_h-\widehat{\bld w}_h))
  \,
  \mathrm{ds}\nonumber,
\end{align}
where $\Pi_h$ denotes the $L^2(\Eh)$-projection onto the tangential 
facet finite element space $\widehat{\bld V}_h^{k-1}$.
Efficient implementation of this local projector $\Pi_h$ was discussed in
\cite[Section 2.2.2]{LehrenfeldSchoberl16}.
Here $\alpha>0$ is a sufficiently large stabilization parameter that
ensures the following coercivity result:
\begin{align}
  \label{coercivity}
A_h^{i}((\bld v_h, \widehat{\bld v}_h),
(\bld v_h, \widehat{\bld v}_h))
\ge \frac12
\sum_{K\in\Th^i} 
\left(\|\bld D(\bld v_h)\|_K^2 
+\frac{\alpha k^2}{h}
\|\Pi_h(\mathsf{tang}(\bld v_h-\widehat{\bld v}_h))\|_{\partial K}^2
\right),
\end{align}
where $\|\cdot\|_S$ indicates the $L^2$-norm on the domain $S$.
A sufficient condition on $\alpha$ that guarantees the above coercivity result was 
presented in \cite[Lemma 1]{AinsworthFu18}.
We take $\alpha = 8$ in our numerical experiments in Section \ref{sec:num}.

The following two results show 
consistency and stability of the semi-discrete scheme \eqref{semi}.
\begin{lemma}[Galerkin-orthogonality for the semi-discrete scheme]
Let $(\bld u, p^f, \bld \eta^s)\in H^2(\Omega)\times 
H^1(\Omega^f)\times H^2(\Omega^s)$ be the solution to the 
model problem \eqref{model}.
Then, the equations \eqref{semi} holds true with 
$(\bld u_h, \widehat{\bld u}_h, 
p^f_h, \bld \eta^s_h, \widehat{\bld \eta}^s_h)$
replaced by 
$(\bld u, \bld u|_{\Eh}, 
p^f, \bld \eta^s, \bld \eta^s|_{\Eh^s})$. That is, we have 
\begin{subequations}
  \label{semiX}
  \begin{alignat}{2}
  \label{semiX-1}
    (\rho \partial_t\bld u, \bld v_h)+ 
    2\mu^f A_h^f(
    (\bld u, \widehat{\bld u}),(\bld v_h, \widehat{\bld v}_h) )
    -(p^f, \nabla\cdot \bld v_h)_f
    -(\nabla\cdot\bld u, q_h^f)_f\;\;
    &\\
   + 2\mu^s A_h^s(
 (\bld \eta^s, \widehat{\bld \eta}^s),(\bld v_h, \widehat{\bld v}_h) )
    +\lambda^s(\nabla \cdot \bld \eta^s, \nabla \cdot\bld v_h)_s
    &\;
    =\;  (\bld f, \bld v_h), \nonumber\\
  \label{semiX-2}
    (\partial_t \bld \eta^s-\bld u, \bld \xi_h^s)_s &\;= \;0,\\
  \label{semiX-3}
    \langle\partial_t \widehat{\bld \eta}^s-\widehat{\bld u}, 
    \widehat{\bld \xi}^s_h\rangle_s&\;=\;0,
  \end{alignat}
  for all 
$(\bld v_h, \widehat{\bld v}_h, q_h^f, \bld \xi_h^s, \widehat{\bld
\xi}_h^s)\in \bld V_{h,0}^k\times \widehat{\bld V}_{h,0}^{k-1}\times 
Q_h^{k-1,f}\times\bld V_{h,0}^{k,s}\times \widehat{\bld V}_{h,0}^{k-1,s} 
$, where $\widehat{\bld u}=\bld u|_{\Eh}$ and 
$\widehat{\bld \eta}^s = \bld \eta^s|_{\Eh^s}$.
\end{subequations}
\end{lemma}
\begin{proof}
  The equations \eqref{semiX-2} and \eqref{semiX-3} 
  follows from the second equation in \eqref{s-eq}.
  We are left to prove the equation \eqref{semiX-1}.
  Since $\mathsf{tang}(\bld u-\widehat{\bld u}) = 0$, we have,
for any function $(\bld v_h, \widehat{\bld v}_h)\in 
\bld V_{h,0}^{k}\times \widehat{\bld V}_{h,0}^{k-1}$, 
\begin{align*}
 A_h^f(
 (\bld u, \widehat{\bld u}),(\bld v_h, \widehat{\bld v}_h) )
 = &\;
  \sum_{K\in\Th^f}\;\int_{K}\bld D(\bld u):
  \bld D(\bld v_h)\,\mathrm{dx}
  -\int_{\partial K}
  \bld D(\bld u)\bld n\cdot \mathsf{tang}(\bld v_h-\widehat{\bld v}_h)\,
  \mathrm{ds}\\
 = &\;
 -(\nabla\cdot\bld D(\bld u), \bld v_h)_f
 +\sum_{K\in\Th^f} 
 \int_{\partial K}
 \bld D(\bld u)\bld n\cdot ((\bld v_h\cdot \bld n) \bld
 n+\mathsf{tang}(\widehat{\bld v}_h))\,
  \mathrm{ds}\\
 = &\;
 -(\nabla\cdot\bld D(\bld u), \bld v_h)_f
 + 
 \int_{\Gamma_h}
 \bld D(\bld u)\bld n^f\cdot ((\bld v_h\cdot \bld n) \bld
 n+\mathsf{tang}(\widehat{\bld v}_h))\,
  \mathrm{ds}.
\end{align*}
Similarly, we have, for any function $(\bld v_h, \widehat{\bld v}_h)\in 
\bld V_{h,0}^{k}\times \widehat{\bld V}_{h,0}^{k-1}$, 
\begin{align*}
 A_h^s(
 (\bld \eta^s, \widehat{\bld \eta^s}),(\bld v_h, \widehat{\bld v}_h) )
 = &\;
 -(\nabla\cdot\bld D(\bld \eta^s), \bld v_h)_s
 + \int_{\Gamma_h}
 \bld D(\bld \eta^s)\bld n^s\cdot ((\bld v_h\cdot \bld n) \bld
 n+\mathsf{tang}(\widehat{\bld v}_h))\,
  \mathrm{ds},\\
  (p^f, \nabla\cdot \bld v_h)_f = 
   &\;-(\nabla p^f, \bld v_h)_f + 
   \int_{\Gamma_h} p^f (\bld v_h\cdot \bld n^f)\,\mathrm{ds},\\
   (\nabla\cdot\bld \eta^s, \nabla\cdot \bld v_h)_s = 
   &\;-(\nabla (\nabla\cdot \bld \eta^s), \bld v_h)_s + 
   \int_{\Gamma_h} (\nabla\cdot \bld \eta^s) 
   (\bld v_h\cdot \bld
   n^s)\,\mathrm{ds}.
\end{align*}
Combining these equations, we get
\begin{align*}
    &(\rho \partial_t\bld u, \bld v_h)+ 
    2\mu^f A_h^f(
    (\bld u, \widehat{\bld u}),(\bld v_h, \widehat{\bld v}_h) )
    -(p^f, \nabla\cdot \bld v_h)_f
    -(\nabla\cdot\bld u, q_h^f)_f\;\;
    \\
    & + 2\mu^s A_h^s(
 (\bld \eta^s, \widehat{\bld \eta}^s),(\bld v_h, \widehat{\bld v}_h) )
    +\lambda^s(\nabla \cdot \bld \eta^s, \nabla \cdot\bld v_h)_s
    -(\bld f, \bld v_h)
    \;\\
    & =\; (\rho^f\partial_t \bld u^f-\nabla\cdot \bld \sigma^f-\bld f^f, \bld v_h)_f 
    +(\rho^s\partial_t\bld u^s-\nabla\cdot \bld \sigma^s-\bld f^s, \bld v_h)_s
    \\
    &
   \;\;\;\; +\int_{\Gamma_h}(\bld \sigma^f\bld n^f+\bld \sigma^s\bld n^s)
    ( (\bld v_h\cdot \bld
   n)\cdot \bld n+\mathsf{tang}(\widehat{\bld v}_h)))\,\mathrm{ds}
  \\
    &=0,
\end{align*}
where 
we used the PDE \eqref{f-eq}, \eqref{s-eq}, and the 
dynamic interface condition in \eqref{interface}.
This completes the proof of the equation $\eqref{semiX-1}$.
\end{proof}

\begin{lemma}[Stability for the semi-discrete scheme]
  \label{lemma:stab-semi}
  Let $(\bld u_h, \widehat{\bld u}_h, p_h^f, \bld \eta_h^s, 
\widehat{\bld \eta}_h^s)\in 
\bld V_{h,0}^k\times \widehat{\bld V}_{h,0}^{k-1}\times 
Q_h^{k-1,f}\times\bld V_{h,0}^{k,s}\times \widehat{\bld V}_{h,0}^{k-1,s} 
$ be the numerical solution to the semi-discrete scheme \eqref{semi}.
Then, the velocity approximation on the fluid domain is exactly divergence free: 
  \begin{align}
    \label{div1}
    \nabla\cdot \bld u_h|_{\Th^f} = 0,
  \end{align}
  and the following energy identity holds:
  \begin{align}
    \label{ener1}
    \frac12\frac{d}{dt} E_h = -2\mu^f 
    A_h^f((\bld u_h, \widehat{\bld u}_h),(\bld u_h, \widehat{\bld u}_h))
    +(\bld f, \bld u_h),
  \end{align} where 
  $
  E_h := (\rho\bld u_h, \bld u_h)+
  \lambda^s(\nabla \cdot \bld \eta_h^s, \nabla \cdot \bld \eta^s_h)_s
 +2\mu^s
    A_h^s((\bld \eta_h, \widehat{\bld \eta}_h),(\bld \eta_h, \widehat{\bld
    \eta}_h))
  $
  is the total energy.
\end{lemma}
\begin{proof}
  Let us first prove the divergence-free property \eqref{div1}.
  By the choice of the velocity finite element space $\bld V_h^{k}$ and
  fluid pressure finite element space $Q_h^{k-1, f}$, we have 
  $\nabla\cdot\bld u_h|_{\Th^f}\in Q_h^{k-1, f}$. Now taking
  $q_h^f=\nabla\cdot\bld u_h|_{\Th^f}$ in equation \eqref{semi-1}, we get
  $$
  (\nabla\cdot\bld u_h, \nabla\cdot\bld u_h)_f=0.
  $$ Hence the divergence-free property \eqref{div1} holds true.
  
  Next, let us prove the energy identity \eqref{ener1}.
  Taking test function 
  $(\bld v_h, \widehat{\bld v}_h) = (\bld u_h, \widehat{\bld u}_h)$
  in equation \eqref{semi-1}, and using the divergence-free property
  \eqref{div1}, we get
\begin{align*}
    (\rho \partial_t\bld u_h, \bld u_h)+ 
    2\mu^f A_h^f(
    (\bld u_h, \widehat{\bld u}_h),(\bld u_h, \widehat{\bld u}_h) )
   + 2\mu^s A_h^s(
 (\bld \eta^s_h, \widehat{\bld \eta}^s_h),(\bld u_h, \widehat{\bld u}_h) )
    +\lambda^s(\nabla \cdot \bld \eta_h^s, \nabla \cdot\bld u_h)_s
    &\;
    =\;  (\bld f, \bld v_h).
\end{align*}
Since $\bld u_h|_{\Th^s}\in \bld V_h^{k,s}$ and 
$\widehat{\bld u}_h|_{\Eh^s}\in \widehat{\bld V}_h^{k-1,s}$, 
equations \eqref{semi-2}--\eqref{semi-3} implies that 
$$
\bld u_h|_{\Th^s} = \partial_t \bld \eta_h^s, \text{ and }  
\widehat{\bld u}_h|_{\Eh^s} = \partial_t {\widehat{\bld \eta}}_h^s.
$$
Hence, 
$$
2\mu^s A_h^s(
 (\bld \eta^s_h, \widehat{\bld \eta}^s_h),(\bld u_h, \widehat{\bld u}_h) )
=
2\mu^s 
A_h^s(
 (\bld \eta^s_h, \widehat{\bld \eta}^s_h),(\partial_t\bld \eta^s_h,
 \partial_t\widehat{\bld \eta}_h^s) )
 =\frac12\frac{d}{dt}(2\mu^s
A_h^s(
 (\bld \eta^s_h, \widehat{\bld \eta}^s_h),(\bld \eta^s_h,
 \widehat{\bld \eta}_h^s))
 ),
$$
and 
$$
\lambda^s(\nabla \cdot \bld \eta_h^s, \nabla \cdot\bld u_h)_s
= 
\frac12\frac{d}{dt}\left(\lambda^s(\nabla \cdot \bld \eta_h^s, \nabla
\cdot\bld \eta_h^s)_s\right).
$$
Combining the above equations, we arrive at the energy identity
\eqref{ener1}. This completes the proof.
\end{proof}

\subsection{Monolithlic fully discrete divergence-conforming HDG scheme}
In this subsection, we consider the temporal discretization of the
semi-discrete scheme \eqref{semi}. 
We propose to use the second-order Crank-Nicolson
scheme. 
For any positive integer $j\in \mathbb{Z}_+$, let
$(\bld u_h^{j-1}, 
\bld \eta_h^{s,j-1}, \widehat{\bld
\eta}_h^{s,j-1})\in \bld V_{h,0}^k \times\bld V_{h,0}^{k,s}\times
\widehat{\bld V}_{h,0}^{k-1,s} 
$ be the numerical solution at time $t_{j-1}$.
Give the time step size $\delta t_{j-1}>0$, 
we proceed to find the solution 
$(\bld u_h^{j}, 
\bld \eta_h^{s,j}, \widehat{\bld
\eta}_h^{s,j})\in \bld V_{h,0}^k \times\bld V_{h,0}^{k,s}\times
\widehat{\bld V}_{h,0}^{k-1,s} 
$ at time $t_{j}=t_{j-1}+\delta t_{j-1}$ along with the solution
$(\widehat{\bld u}_h^{j-1/2}, 
p_h^{f,j-1/2})\in \widehat{\bld V}_{h,0}^{k-1} \times
Q_{h}^{k-1,f}
$ at time $t_{j-1/2}=t_{j-1}+\frac12\delta t_{j-1}$ such that the following
equations holds:
\begin{subequations}
  \label{full}
  \begin{alignat}{2}
  \label{full-1}
  (\rho \frac{\bld u_h^{j}-\bld u_h^{j-1}}{\delta t_{j-1}}, \bld v_h)+ 
    2\mu^f A_h^f(
    (\bld u_h^{j-1/2}, \widehat{\bld u}_h^{j-1/2}),(\bld v_h, \widehat{\bld v}_h) )
    -(p_h^{f,j-1/2}, \nabla\cdot \bld v_h)_f
    &\\
\;\;    -(\nabla\cdot\bld u_h^{j-1/2}, q_h^f)_f\;\;
    + 2\mu^s A_h^s(
   (\bld \eta^{s,j-1/2}_h, \widehat{\bld \eta}^{s,j-1/2}_h),(\bld v_h, \widehat{\bld v}_h) )
   +\lambda^s(\nabla \cdot \bld \eta_h^{s,j-1/2}, 
   \nabla \cdot\bld v_h)_s
    &\;
    =\;  (\bld f^{j-1/2}, \bld v_h), \nonumber\\
  \label{full-2}
    (\frac{\bld \eta_h^{s,j}-\bld \eta_h^{s,j-1}}{\delta t_{j-1}}-\bld
    u_h^{j-1/2}, \bld \xi_h^s)_s &\;= \;0,\\
  \label{full-3}
    \langle
    \frac{\widehat{\bld \eta}_h^{s,j}-\widehat{\bld \eta}_h^{s,j-1}}{\delta t_{j-1}}-
    \widehat{\bld u}_h^{j-1/2}, 
    \widehat{\bld \xi}^s_h\rangle_s&\;=\;0,
  \end{alignat}
  for all 
$(\bld v_h, \widehat{\bld v}_h, q_h^f, \bld \xi_h^s, \widehat{\bld
\xi}_h^s)\in \bld V_{h,0}^k\times \widehat{\bld V}_{h,0}^{k-1}\times 
Q_h^{k-1,f}\times\bld V_{h,0}^{k,s}\times \widehat{\bld V}_{h,0}^{k-1,s} 
$,
where 
$$
\bld u_h^{j-1/2}:=\frac12(\bld u_h^j+\bld u_h^{j-1}),\;\; 
\bld \eta_h^{s,j-1/2}:=\frac12(\bld \eta_h^{s,j}+\bld \eta_h^{s,j-1}), \;\;
\widehat{\bld \eta}_h^{s,j-1/2}:=\frac12(
\widehat{\bld \eta}_h^{s,j}+\widehat{\bld \eta}_h^{s,j-1}), 
$$
\end{subequations}

We have the following result on the energy stability of the 
fully discrete scheme \eqref{full}.
\begin{lemma}[Stability for the fully discrete scheme]
  \label{lemma:stab-full}
Let $(\bld u_h^{0}, 
\bld \eta_h^{s,0}, \widehat{\bld
\eta}_h^{s,0})\in \bld V_{h,0}^k \times\bld V_{h,0}^{k,s}\times
\widehat{\bld V}_{h,0}^{k-1,s} 
$ be a proper projection of the initial data 
in \eqref{init} such that 
$\nabla\cdot \bld u_h^0|_{\Th^f} = 0$.
  For any positive integer $j\in \mathbb{Z}_+$,  let $(\bld u_h^{j},
  \widehat{\bld u}_h^{j-1/2}, p_h^{f,j-1/2}, \bld \eta_h^{s,j}, 
  \widehat{\bld \eta}_h^{s,j})\in 
\bld V_{h,0}^k\times \widehat{\bld V}_{h,0}^{k-1}\times 
Q_h^{k-1,f}\times\bld V_{h,0}^{k,s}\times \widehat{\bld V}_{h,0}^{k-1,s} 
$ be the numerical solution to the fully discrete scheme \eqref{full}.
Then, the velocity approximation on the fluid domain is exactly divergence free: 
  \begin{align}
    \label{div2}
    \nabla\cdot \bld u_h^j|_{\Th^f} = 0,
  \end{align}
  and the following energy identity holds true: 
  \begin{align}
    \label{ener2}
    \frac12\frac{E_h^j-E_h^{j-1}}{\delta t_{j-1}} = -2\mu^f 
    A_h^f((\bld u_h^{j-1/2}, \widehat{\bld u}_h^{j-1/2}),(\bld u_h^{j-1/2},
    \widehat{\bld u}_h^{j-1/2}))
    +(\bld f^{j-1/2}, \bld u_h),
  \end{align} where 
  $
  E_h^j := (\rho\bld u_h^j, \bld u_h^j)+
  \lambda^s(\nabla \cdot \bld \eta_h^{s,j}, \nabla \cdot \bld \eta^{s,j}_h)_s
 +2\mu^s
 A_h^s((\bld \eta^{s,j}_h, \widehat{\bld \eta}^{s,j}_h),(\bld \eta^{s,j}_h, \widehat{\bld
 \eta}^{s,j}_h))
  $
  is the total energy at time $t_j$.
\end{lemma}
\begin{proof}
  The proof follows the same line as those for the semi-discrete case in
  Lemma \ref{lemma:stab-semi}, which we omit for simplicity. 
\end{proof}

\subsubsection{Efficient implementation of the fully discrete
scheme \eqref{full}}
\label{sub:imp}
We now present an efficient implementation of the scheme 
\eqref{full} whose globally coupled linear system consists of 
DOFs for 
the normal component of the velocity approximation, 
the tangential component of the hybrid velocity approximation
on the mesh skeleton, and one pressure DOF per element
on the mesh.

We need the following result on the characterization of the 
fully discrete solution.
\begin{lemma}[Characterization of the fully discrete solution]
  \label{lemma:char}
  \begin{subequations}
  For any positive integer $j\in \mathbb{Z}_+$,  let $(\bld u_h^{j},
  \widehat{\bld u}_h^{j-1/2}, p_h^{f,j-1/2}, \bld \eta_h^{s,j}, 
  \widehat{\bld \eta}_h^{s,j})\in 
\bld V_{h,0}^k\times \widehat{\bld V}_{h,0}^{k-1}\times 
Q_h^{k-1,f}\times\bld V_{h,0}^{k,s}\times \widehat{\bld V}_{h,0}^{k-1,s} 
$ be the numerical solution to the fully discrete scheme \eqref{full}.
Then, 
$(\bld u_h^{j-1/2}, \widehat{\bld u}_h^{j-1/2},p_h^{f,j-1/2})
\in 
\bld V_{h,0}^k\times \widehat{\bld V}_{h,0}^{k-1}\times 
Q_h^{k-1,f}
$ is the unique solution to the following equations:
  \begin{alignat}{2}
  \label{fullX}
  &(2\rho \frac{\bld u_h^{j-1/2}}{\delta t_{j-1}}, \bld v_h)+ 
    2\mu^f A_h^f(
    (\bld u_h^{j-1/2}, \widehat{\bld u}_h^{j-1/2}),(\bld v_h, \widehat{\bld
    v}_h) )\\
  &    -(p_h^{f,j-1/2}, \nabla\cdot \bld v_h)_f
-(\nabla\cdot\bld u_h^{j-1/2}, q_h^f)_f
\;\;\nonumber\\
  &
+\frac12\delta t_{j-1}\left( 2\mu^s A_h^s(
   (\bld u^{j-1/2}_h, \widehat{\bld u}^{j-1/2}_h),(\bld v_h, \widehat{\bld v}_h) )
   +\lambda^s(\nabla \cdot \bld u_h^{j-1/2}, 
 \nabla \cdot\bld v_h)_s\right)
    \;
    =\;  
    F^{j-1/2}((\bld v_h, \widehat{\bld v}_h))\nonumber
  \end{alignat}
  for all 
$(\bld v_h, \widehat{\bld v}_h, q_h^f)\in \bld V_{h,0}^k\times \widehat{\bld V}_{h,0}^{k-1}\times 
Q_h^{k-1,f}$, where the right hand side 
\begin{align}
  \label{source}
    F^{j-1/2}((\bld v_h, \widehat{\bld v}_h))
    =&\;
    (2\rho \frac{\bld u_h^{j-1}}{\delta t_{j-1}}, \bld v_h)
    +(\bld f^{j-1/2}, \bld v_h)\\
     &\;
    - \left( 2\mu^s A_h^s(
   (\bld \eta^{s,j-1}_h, \widehat{\bld \eta}^{s,j-1}_h),(\bld v_h, \widehat{\bld v}_h) )
   +\lambda^s(\nabla \cdot \bld \eta_h^{s,j-1}, 
 \nabla \cdot\bld v_h)_s\right),\nonumber
\end{align}
where $\bld f^{j-1/2}$ is the source term evaluated at time 
$t_{j-1/2}$.
Moreover, the velocity and displacement approximations at time 
$t_j$ satisfy the following relations:
\begin{align}
  \label{u-recov}
  \bld u_h^{j} = &\; 2\bld u_h^{j-1/2}-\bld u_h^{j-1},
  \;\;
  \bld \eta_h^{s,j} = \; 
  \bld \eta_h^{s,j-1}+\delta t_{j-1}\bld u_h^{j-1/2}|_{\Th^s},\;\;
  \widehat{\bld \eta}_h^{s,j} = \; 
  \widehat{ \bld \eta}_h^{s,j-1}+\delta t_{j-1}\widehat{\bld
  u}_h^{j-1/2}|_{\Th^s}.
\end{align}
  \end{subequations}
\end{lemma}
\begin{proof}
  The relations \eqref{u-recov} are direct consequences
  of the definition of $\bld u_h^{j-1/2}$, the same choice of the finite
  element spaces for velocity and displacement, and 
  the equations \eqref{full-2}--\eqref{full-3}.
  Plugging in these relations back to the equations \eqref{full-1}, and
  reordering the terms, we recover the equations \eqref{fullX}.
  This completes the proof.
\end{proof}

\begin{remark}[Connection with the coupled momentum method]
The idea of using the same finite element space
for displacement and velocity approximations
to
eliminate the displacement unknowns in the global linear system was
originated in the coupled momentum method of 
Figueroa et al. \cite{figueroa2006coupled}, 
where they considered an FSI problem with thin structure.
See also related work in \cite{NobileVergara08}.
\end{remark}

With the help of Lemma \ref{lemma:char}, we proceed to implement 
the fully discrete scheme \eqref{full} as follows:
Let $(\bld u_h^{0}, 
\bld \eta_h^{s,0}, \widehat{\bld
\eta}_h^{s,0})\in \bld V_{h,0}^k \times\bld V_{h,0}^{k,s}\times
\widehat{\bld V}_{h,0}^{k-1,s} 
$ be a proper projection of the initial data 
in \eqref{init}.
For $j = 1,2,\cdots$, we proceed the following three steps to 
advance solution from time level $t_{j-1}$ to time level $t_{j}=t_{j-1}+\delta
t_{j-1}$:
\begin{itemize}
  \item [(1)] Determine the time step size 
    $\delta t_{j-1}$ and compute the right hand side
    $F^{j-1/2}$ in \eqref{source}.
  \item [(2)] 
    Solve 
    for 
$(\bld u_h^{j-1/2}, \widehat{\bld u}_h^{j-1/2},p_h^{f,j-1/2})
\in 
\bld V_{h,0}^k\times \widehat{\bld V}_{h,0}^{k-1}\times 
Q_h^{k-1,f}
$ using equations \eqref{fullX}.
\item [(3)] Recover velocity and displacement approximations
  at time $t_{j}$ using the relations \eqref{u-recov}.
\end{itemize}

The major computational cost of the above implementation lies in 
the global linear system solver in step (2).
To make the linear system problem easier to solve, we
introduce an equivalent 
characterization of the solution to the equations \eqref{fullX}
in Lemma \ref{lemma:mod} below. In the actual implementation, 
we solve the equivalent linear system problem \eqref{fullY} in step (2) 
instead of \eqref{fullX}.
In the next section, we will design
an efficient block AMG preconditioner for this system.

\begin{lemma}[A modified implementation of the scheme \eqref{fullX}]
  \label{lemma:mod}
For any positive integer $j\in \mathbb{Z}_+$,  
  let 
$(\bld u_h^{j-1/2}, \widehat{\bld u}_h^{j-1/2},p_h^{j-1/2})
\in 
\bld V_{h,0}^k\times \widehat{\bld V}_{h,0}^{k-1}\times 
Q_h^{k-1}$ be the unique solution to the following 
equations:
  \begin{alignat}{2}
  \label{fullY}
  &(2\rho \frac{\bld u_h^{j-1/2}}{\delta t_{j-1}}, \bld v_h)+ 
    2\mu^f A_h^f(
    (\bld u_h^{j-1/2}, \widehat{\bld u}_h^{j-1/2}),(\bld v_h, \widehat{\bld
    v}_h) )
+\delta t_{j-1}\mu^s A_h^s(
   (\bld u^{j-1/2}_h, \widehat{\bld u}^{j-1/2}_h),(\bld v_h, \widehat{\bld v}_h) )
    \\
  &    -(p_h^{j-1/2}, \nabla\cdot \bld v_h)
-(\nabla\cdot\bld u_h^{j-1/2}, q_h)
-(\frac{2}{\delta t_{j-1}\lambda^s}p_h^{j-1/2}, q_h)_s
\;\;=\;  
    F^{j-1/2}((\bld v_h, \widehat{\bld v}_h))\nonumber
  \end{alignat}
  for all 
$(\bld v_h, \widehat{\bld v}_h, q_h)\in \bld V_{h,0}^k\times \widehat{\bld V}_{h,0}^{k-1}\times 
Q_h^{k-1}$, where the right hand side $F^{j-1/2}$
is defined in \eqref{source}.
Then 
$(\bld u_h^{j-1/2}, \widehat{\bld u}_h^{j-1/2},p_h^{j-1/2}|_{\Th^f})
\in 
\bld V_{h,0}^k\times \widehat{\bld V}_{h,0}^{k-1}\times 
Q_h^{k-1,f}$  is the unique 
solution to the equations \eqref{fullX}.
\end{lemma}
\begin{proof}
  Taking test function $q_h\in Q_h^{k-1,s}$ in equations \eqref{fullY}, we get
$$
(\nabla \cdot \bld u_h^{j-1/2}+\frac{2}{\delta t_{j-1}\lambda^s}
p_h^{j-1/2}, 
q_h)_s = 0 \quad \forall q_h\in Q_h^{k-1,s}.
$$
Since $\nabla\cdot \bld u_h^{j-1/2}|_{\Th^s} \in Q_h^{k-1,s}$,
the above equation implies that 
$$
p_h|_{\Th^s} = -\frac{\delta t_{j-1}\lambda^s}{2}\nabla\cdot \bld
u_h^{j-1/2}|_{\Th^s}.
$$
Hence, 
$$
-(p_h^{j-1/2}, \nabla \cdot \bld v_h)
=
-(p_h^{j-1/2}, \nabla \cdot \bld v_h)_f
+\frac{\delta t_{j-1}}{2}\lambda^s(\nabla\cdot\bld u_h^{j-1/2}, \nabla
\cdot \bld v_h)_s.$$
Plugging this expression back to the equations \eqref{fullY}, we recover
the equations in \eqref{fullX} for all 
$(\bld v_h, \widehat{\bld v}_h, q_h)\in \bld V_{h,0}^k\times \widehat{\bld V}_{h,0}^{k-1}\times 
Q_h^{k-1,f}$. This completes the proof.
\end{proof}

\begin{remark}[Other high-order 
  implicit time stepping strategies]
 We concentrated on the discretization and implementation of the 
 Crank-Nicolson time stepping \eqref{full} in this subsection.
 Alternatively, one can use any other high-order implicit time stepping
 strategies, like the backward different formula (BDF) or the 
 diagonally implicit Runge-Kutta methods \cite{HairerWanner10}.
 The third-order BDF3 scheme reads as follows (assuming uniform time step
 size $\delta t>0$):
 For $j\ge 3$, 
given approximations 
$(\bld u_h^{j-m}, 
\bld \eta_h^{s,j-m}, \widehat{\bld
\eta}_h^{s,j-m})\in \bld V_{h,0}^k \times\bld V_{h,0}^{k,s}\times
\widehat{\bld V}_{h,0}^{k-1,s} 
$ at time $t_{j-m}=(j-m)\,\delta t$ 
for $m=1,2,3$, 
we proceed to find the solution 
$(\bld u_h^{j},\widehat{\bld u}_h^j, q_h^{f,j}, 
\bld \eta_h^{s,j}, \widehat{\bld
\eta}_h^{s,j})\in \bld V_{h,0}^k \times\widehat{\bld V}_{h,0}^{k-1}\times Q_h^f\times
\bld V_{h,0}^{k,s}\times
\widehat{\bld V}_{h,0}^{k-1,s} 
$ at time $t_{j}=j\,\delta t$ 
 such that the following
equations holds:
\begin{subequations}
  \label{fullB}
  \begin{alignat}{2}
  \label{fullB-1}
  (\rho \mathtt{D}_t\bld u_h^j, \bld v_h)+ 
    2\mu^f A_h^f(
    (\bld u_h^{j}, \widehat{\bld u}_h^{j}),(\bld v_h, \widehat{\bld v}_h) )
    -(p_h^{f,j}, \nabla\cdot \bld v_h)_f
    &\\
\;\;    -(\nabla\cdot\bld u_h^{j}, q_h^f)_f\;\;
    + 2\mu^s A_h^s(
   (\bld \eta^{s,j}_h, \widehat{\bld \eta}^{s,j}_h),(\bld v_h, \widehat{\bld v}_h) )
   +\lambda^s(\nabla \cdot \bld \eta_h^{s,j}, 
   \nabla \cdot\bld v_h)_s
    &\;
    =\;  (\bld f^{j}, \bld v_h), \nonumber\\
  \label{fullB-2}
    (
    \mathtt{D}_t\bld \eta_h^{s,j}
    -\bld
    u_h^{j-1/2}, \bld \xi_h^s)_s &\;= \;0,\\
  \label{fullB-3}
    \langle
    \mathtt{D}_t\widehat{\bld \eta}_h^{s,j}
    \widehat{\bld u}_h^{j-1/2}, 
    \widehat{\bld \xi^s}_h\rangle_s&\;=\;0,
  \end{alignat}
  for all 
$(\bld v_h, \widehat{\bld v}_h, q_h^f, \bld \xi_h^s, \widehat{\bld
\xi}_h^s)\in \bld V_{h,0}^k\times \widehat{\bld V}_{h,0}^{k-1}\times 
Q_h^{k-1,f}\times\bld V_{h,0}^{k,s}\times \widehat{\bld V}_{h,0}^{k-1,s} 
$,
where 
$$
\mathtt{D}_t\bld \phi^{j}:=\frac{1}{\delta t}
\left(\frac{11}{6}\phi^j-3\phi^{j-1}+\frac32 \phi^{j-2}
-\frac13\phi^{j-3}\right)
$$
is the third order BDF 
discretization of the time derivative $\partial_t \phi$.
We can proceed along the same lines as in Subsection 
\ref{sub:imp} to implement the scheme \eqref{fullB} such that 
we only need to solve a global linear system 
of the form \eqref{fullY} in each time step.
\end{subequations}

\end{remark}
\section{Preconditioning}
\label{sec:prec}
\subsection{Preliminaries}
In this section, we concentrate ourselves to the efficient solver for the
linear system problem \eqref{fullY}.
The same technique can be used to solve the related linear system 
for the scheme based on BDF3 time stepping \eqref{fullB}.
To simplify notation, we remove all temporal indices in this section.
Hence the linear system problem we are interested in have the following specific
form: Find  
$(\bld u_h, \widehat{\bld u}_h,p_h)
\in 
\bld V_{h,0}^k\times \widehat{\bld V}_{h,0}^{k-1}\times 
Q_h^{k-1}$ such that 
  \begin{alignat}{2}
  \label{fullZ}
  &(2\rho \frac{\bld u_h}{\delta t}, \bld v_h)+ 
    2\mu^f A_h^f(
    (\bld u_h, \widehat{\bld u}_h),(\bld v_h, \widehat{\bld
    v}_h) )
+\delta t\mu^s A_h^s(
   (\bld u_h, \widehat{\bld u}_h),(\bld v_h, \widehat{\bld v}_h) )
    \\
  &    -(p_h, \nabla\cdot \bld v_h)
-(\nabla\cdot\bld u_h, q_h)
-(\frac{2}{\delta t\,\lambda^s}p_h, q_h)_s
\;\;=\;  
    F((\bld v_h, \widehat{\bld v}_h))\nonumber
  \end{alignat}
  for all 
$(\bld v_h, \widehat{\bld v}_h, q_h)\in \bld V_{h,0}^k\times \widehat{\bld V}_{h,0}^{k-1}\times 
Q_h^{k-1}$. Note that all the 
finite element spaces are defined on the whole domain $\Omega$.
To further simplify notation, we denote 
$$\mu:= \left\{\begin{tabular}{ll}
    $\mu^f$ & on $\Omega^f$\\
    $0.5 \delta t\,\mu^s$ & on $\Omega^s$
\end{tabular}\right.,
\text{ and }
\gamma:= \left\{\begin{tabular}{ll}
    $0$ & on $\Omega^f$\\
    $2/\delta t/\lambda^s$ & on $\Omega^s$
\end{tabular}\right..
$$
We also denote 
$$
A_h^{\mu}((\bld u_h, \widehat{\bld u}_h),(\bld v_h, \widehat{\bld
    v}_h) ):=
    2\mu^f A_h^f(
    (\bld u_h, \widehat{\bld u}_h),(\bld v_h, \widehat{\bld
    v}_h) )
+\delta t\mu^s A_h^s(
   (\bld u_h, \widehat{\bld u}_h),(\bld v_h, \widehat{\bld v}_h)),
$$
which is an HDG discretization of the variable coefficient diffusion
operator $-\nabla\cdot(\mu\bld D(\bld u))$ on the whole domain $\Omega$.
Hence, the formulation \eqref{fullZ} simplifies to 
\begin{align}
  \label{fullF}
  &\frac{2}{\delta t}(\rho \bld u_h, \bld v_h)+ 
  A_h^{\mu}(
    (\bld u_h, \widehat{\bld u}_h),(\bld v_h, \widehat{\bld
    v}_h) )
  -(p_h, \nabla\cdot \bld v_h)
-(\nabla\cdot\bld u_h, q_h)
-(\gamma p_h, q_h)
=    F((\bld v_h, \widehat{\bld v}_h)).
\end{align}
The problem \eqref{fullF} can be rewritten in a matrix-vector formulation:
Find $[\underline{\mathsf{u}}_h; \mathsf{p}_h]\in  
\mathbb{R}^{N_u+N_p}$ such that 
\begin{align}
  \label{matrix}
  \left[\begin{tabular}{cc}
      $\mathsf{A}_{h}^\mu+\frac{2}{\delta t} 
      \mathsf{M}^{\rho}_h$ & $\mathsf{B}_h$ \\[.8ex]
    $\mathsf{B}_h^T$ & $-\mathsf{M}^{\gamma}_h$
    \end{tabular} 
  \right]
    \left[
      \begin{tabular}{l}
        $\underline{\mathsf u}_h$\\
      $\mathsf p_h$
    \end{tabular}
    \right]
    =\left[
      \begin{tabular}{l}
        $\mathsf{F}$\\
      $0$
    \end{tabular}
    \right],
\end{align}
where $\underline{\mathsf u}_h\in \mathbb{R}^{N_u}
$ is the coefficient vector for the 
compound velocity approximation 
$ \underline{\bld u}_h:= (\bld u_h, \widehat{\bld u}_h)
\in \bld V_{h,0}^k\times 
\widehat{\bld V}_{h,0}^{k-1}$ with 
$N_u$ being the dimension of the compound finite element space
$ \bld V_{h,0}^k\times 
\widehat{\bld V}_{h,0}^{k-1}$,
${\mathsf p}_h\in \mathbb{R}^{N_p}
$ is the coefficient vector for the 
pressure approximation
${p}_h\in {Q}_{h}^{k-1}$ with 
$N_p$ being the dimension of the finite element space
${Q}_{h}^{k-1}$. Moreover, 
the matrix $\mathsf{A}_h^\mu\in \mathbb{R}^{N_u\times N_u}$ 
is associated with the bilinear form 
$A_h^\mu(\underline{\bld u}_h, \underline{\bld v}_h)$, 
the matrix $\mathsf{M}_h^\rho\in \mathbb{R}^{N_u\times N_u}$ 
is associated with the bilinear form 
$(\rho \bld u_h, \bld v_h)$, 
the matrix $\mathsf{B}_h\in \mathbb{R}^{N_p\times N_u}$ 
is associated with the bilinear form 
$-(p_h, \nabla\cdot \bld v_h)$, 
the matrix $\mathsf{M}_h^{\gamma}\in \mathbb{R}^{N_p\times N_p}$ 
is associated with the bilinear form 
$(\gamma p_h, q_h)$, and 
the vector $\mathsf{F}\in \mathbb{R}^{N_u}$ is associated 
with the linear form $F(\underline{\bld v}_h)$.
The big matrix in the linear system \eqref{matrix} has a block structure and is
symmetric and indefinite, with the 1-1 block 
      $\mathsf{A}_{h}^\mu+\frac{2}{\delta t} 
      \mathsf{M}^{\rho}_h$ 
      being symmetric positive definite (SPD), 
and the 2-2 block $-\mathsf{M}^{\gamma}_h$ being symmetric and negative semi-definite.

A popular method to solve  the symmetric saddle point problem
\eqref{matrix}, which we adopt in this work,
is to use a preconditioned MinRes solver \cite{Saad03}
with the following block diagonal preconditioner
\cite{Murphy00,Ipsen01}:
\begin{align}
  \label{prec}
  \mathsf{P} = 
  \left[\begin{tabular}{cc}
      $\widehat{\mathsf {iA}}
      $ 
      & $0$ \\[.8ex]
      $0$ & $\widehat{\mathsf{iS}}$
    \end{tabular} 
  \right]
\end{align}
where $\widehat{\mathsf{iA}}$ is an appropriate preconditioner of the SPD matrix 
      $\mathsf A:=\mathsf{A}_{h}^\mu+\frac{2}{\delta t} 
      \mathsf{M}^{\rho}_h,$  
      and 
      $\widehat{\mathsf{iS}}$ is an appropriate preconditioner of the
      (dense) Schur complement SPD matrix 
$\mathsf S :=\mathsf{B}_h^T(\mathsf{A}_{h}^\mu+\frac{2}{\delta t} 
\mathsf{M}^{\rho}_h)^{-1}\mathsf{B}_h
+\mathsf{M}_h^{\gamma}
.$ 
The detailed construction of the 
preconditioner for the Schur complement (pressure) matrix $\mathsf{S}$ is discussed in Subsection \ref{sec:s-block}, where 
we borrow ideas in the literature on preconditioning 
the closely related, generalized Stokes problem \cite{Cahouet88,Elman01,Olshanskii06}.
The detailed construction of the 
preconditioner for the SPD velocity matrix $\mathsf{A}$
is discussed in Subsection \ref{sec:a-block}, where 
we use
an auxiliary space preconditioner \cite{Xu96} 
along with algebraic multigrid.
We mention that for polynomial degree $k\ge 2$, the preconditioned 
MinRes solver is applied to the static condensed subsystem of 
\eqref{matrix}, see the discussion in Remark \ref{rk:condense}.

\begin{remark}[Connection with a generalized Stokes interface problem]
  The discretization \eqref{fullF}, or the form \eqref{matrix},
  is closely related to
  a divergence-conforming  HDG discretization of a
  generalized Stokes interface problem (with a fixed interface)
  with variable density $\rho$ and variable viscosity 
  $\mu$, c.f. 
  \cite{Fu20}. The only difference between the divergence-conforming HDG linear system for 
  the generalized Stokes interface problem and the 
  current FSI problem is that the pressure block is {\it zero} for the
former, while it is $-\mathcal{M}_{h}^{\gamma}$
for the latter, which is a symmetric negative semi-definite 
matrix and 
represents the compressibility of the structure.
A non-zero pressure block also appears in the 
finite element discretization of the Stokes problem
using pressure-stabilized methods, or the linear elasticity problem with 
a  displacement-pressure formulation.
\end{remark}

\begin{remark}[Static condensation for $k\ge 2$]
  \label{rk:condense}
  When polynomial degree $k\ge 2$, we shall 
  solve the linear system problem 
  \eqref{matrix} using static condensation to locally eliminate
  interior  velocity DOFs and high-order pressure DOFs 
  \cite{LehrenfeldSchoberl16}.
The resulting global linear system after static condensation
consists of DOFs for the normal component of velocity
approximation (of degree $k$)
in  $\bld V_{h,0}^k$
and the tangential (hybrid) velocity approximation (of degree $k-1$)
in $\widehat{\bld V}_{h,0}^{k-1}$
on the mesh skeleton (facets), and 
cell-average of pressure approximations (of degree 0) on the mesh.
We denote the compound velocity space corresponding to 
the DOFs on mesh skeleton by $\underline{\bld V}_{h,0}^{k,
\mathsf{gl}}$ which is a subset of the compound space 
$\underline{\bld V}_{h,0}^k$. The global pressure space is 
simply the space of piecewise constants $Q_h^0$.
The linear system after condensation has a similar structure as that in
\eqref{matrix} with a reduced size. We shall apply the precondioned 
MinRes solver for the condensed system in this case.
For this case, the matrices $\mathsf{A}$
and $\mathsf{S}$ shall be understood to be defined on the
reduced spaces $\underline{\bld V}_{h,0}^{k,\mathsf{gl}}$ and $Q_h^{0}$, respectively.
\end{remark}

\subsection{Preconditioning the Schur complement pressure matrix
$\mathsf S$}
\label{sec:s-block}
The preconditioner $\widehat{\mathsf{iS}}$ acts on the 
piecewise constant global pressure space $Q_h^0$, and
is given as follows:
\begin{align}
  \label{s-pre}
  \widehat{\mathsf{iS}}
  =(\mathsf{M}_h^{\mu,\gamma})^{-1}
  +(\mathsf{N}_h^{\rho,\gamma})^{-1}, 
\end{align}
where 
$\mathsf{M_h^{\mu,\gamma}}$ is the
weighted mass matrix associated with the bilinear form 
$((\mu^{-1}+\gamma) p_h, q_h)$ on the piecewise constant 
global pressure space $Q_h^0$, 
and $\mathsf{N}_h^{\rho,\gamma}$ is the matrix
associated with the bilinear form 
\begin{align}
  \label{neumann}
(\gamma p_h, q_h)+\frac{\delta t}{2}
\sum_{F\in\Eh\backslash{\partial \Omega}}\int_{F}\frac{\{\rho^{-1}\}}{h}\jmp{p_h}
\jmp{q_h}\,\mathrm{ds},\quad \forall p_h, q_h\in Q_h^0,
\end{align}
where 
$\{\rho\}|_{F}:=\frac{\rho^++\rho^-}{\rho^+\rho^-}$ is the geometric
average of $\rho$, and 
$\jmp{\phi}=\phi^+-\phi^-$ is the jump of $\phi$ on an interior facet $F$.
Note that the mass matrix $\mathsf{M}^{\mu,\gamma}$ is diagonal and 
its inversion is trivial. 
Also, note that the bilinear form \eqref{neumann}
corresponds to the interior penalty discretization of the operator 
$\gamma p - \frac{\delta t}{2}\nabla \cdot (\rho^{-1}\nabla p)$
with a homogeneous Neumann boundary condition using the piecewise constant space $Q_h^0$. 
The jump term 
in \eqref{neumann} was shown in \cite{Rusten96} to be 
spectrally equivalent to the operator 
$\frac{\delta t}{2}\mathsf{B}_h^T(\mathsf{M}_h^{\rho})^{-1}\mathsf{B}_h$ when
the density $\rho$ is uniformly bounded from above and below.
Hence, $(\mathsf{N}_h^{\rho,\gamma})^{-1}$ serves as a robust
preconditioner for the (dense) Schur complement matrix 
$\frac{\delta
t}{2}\mathsf{B}_h^T(\mathsf{M}_h^{\rho})^{-1}\mathsf{B}_h+\mathsf{M}_h^{\gamma}$.
In the actual numerical realization 
of  $(\mathsf{N}_h^{\rho,\gamma})^{-1}$
, we use the hypre's BoomerAMG preconditioner \cite{hypre, Henson02}
for the matrix $\mathsf{N}_h^{\rho,\gamma}$.

We note that the
pressure Schur complement preconditioner \eqref{s-pre} 
was initially introduced for the generalized Stokes problem (constant
density, constant viscosity, and $\gamma=0$) by 
Cahouet and Chabard \cite{Cahouet88}.
Robustness of this 
Cahouet-Chabard preconditioner 
for the generalized Stokes problem
with respect to variations in the mesh size
$h$ and time step size $\delta t$ was proven in \cite{Bramble97,Mardal04,
Olshanskii06}.
It was then generalized by Olshanskii et al. \cite{Olshanskii06}
to the generalized Stokes interface problem
(variable density, variable viscosity, and $\gamma =0$).
While a theoretical proof of the robustness of the preconditioner in \cite{Olshanskii06} for the variable density and viscosity
case was lacking due to the lack of regularity results for the stationary
Stokes interface problem, numerical results performed in 
\cite{Olshanskii06} seems to indicate that the preconditioner is 
robust also with respect to the jumps in viscosity and density
in large parameter ranges. 
Hence, our preconditioner \eqref{s-pre} can be considered as a 
generalization of the one in \cite{Olshanskii06} to take into account the
structure compressibility ($\gamma>0$ on $\Omega_s$) in the pressure block.

\subsection{Preconditioning the velocity stiffness matrix $\mathsf A$}
\label{sec:a-block}
The matrix $\mathsf{A}$ corresponds to the divergence-conforming 
HDG discretization of the elliptic operator $\frac{2}{\delta t}\rho\bld u
-2\nabla\cdot (\mu\mathbf{D}(\bld u))$. Here we propose to use 
the auxiliary space precondioner \cite{Xu96} developed in \cite{Fu20a}. 
The auxiliary space is the continuous linear Lagrange 
finite elements:
$$
\bld V_{h,0}^{cg} := \{\bld v\in \bld H_0^{1}(\Omega):\;\;
\bld v|_K\in [\mathcal{P}^1(K)]^d, \;\forall K\in \Th\}.
$$
The auxiliary space preconditioner for $\mathsf{A}$ is of the following
form:
\begin{align}
  \label{a-pre}
  \widehat{ \mathsf{iA}} = \mathsf{R} + \mathsf{P}\mathcal{A}^{-1}\mathsf{P}^T,
\end{align}
where $\mathsf{R}\in \mathbb{R}^{N_{u}^{\mathsf{gl}}\times 
N_{u}^{\mathsf{gl}}}$ is the (point) Gauss-Seidel smoother for the matrix
$\mathsf{A}$, with $N_{u}^\mathsf{gl}$
being the dimension of the reduced compound space 
$\underline{\bld V}_{h,0}^{k,\mathsf{gl}}$, 
the matrix $\mathcal{A}\in \mathbb{R}^{N_c\times N_c}
$ is the matrix associated with the 
following bilinear form on the auxiliary space $\bld V_{h,0}^{cg}$:
$$
(\frac{2}{\delta t}\mathtt u_h, \mathtt v_h)
+ 2(\mu\bld D(\mathtt u_h), 
\bld D(\mathtt v_h)),\quad \forall \mathtt u_h,\mathtt v_h\in \bld
V_{h,0}^{cg},
$$
where $N_c$ is the dimension of $\bld V_{h,0}^{cg}$, 
and the matrix 
$\mathsf{P}\in \mathbb{R}^{N_u^{\mathsf{gl}}\times N_c}$
is associated with the projector
$\underline{\Pi}: 
\bld V_{h,0}^{cg}
\rightarrow \underline{\bld V}_{h,0}^{k,\mathsf{gl}} 
$, which is defined as follows: 
for any function 
$\mathtt u_h\in \bld V_{h,0}^{cg}$, find
$\underline{\Pi}\mathtt u_h 
= (\Pi\mathtt u_h, \widehat{\Pi}\mathtt u_h)\in 
\underline{\bld V}_{h,0}^{k,\mathsf{gl}}$ such that
\begin{subequations}
  \label{proj}
\begin{align}
\sum_{F\in\Eh}
\int_F
(\Pi \mathtt u_h\cdot \bld n)(\bld v_h \cdot\bld n)
\mathrm{ds}
=&\;
\sum_{F\in\Eh}
\int_F
(\mathtt u_h\cdot \bld n)(\bld v_h\cdot\bld n)
\mathrm{ds},\\
\sum_{F\in\Eh}
\int_F
\mathsf{tang}(\widehat{\Pi} \mathtt{{u}}_h)\cdot
\mathsf{tang}(\widehat{\bld v}_h)
\mathrm{ds}
=&\;
\sum_{F\in\Eh}
\int_F
\mathsf{tang}(\mathtt u_h)\cdot
\mathsf{tang}(\widehat{\bld v}_h)
\mathrm{ds},
\end{align}
\end{subequations}
for all $(\bld v_h, \widehat{\bld v}_h)\in
\underline{\bld V}_{h,0}^{k,\mathsf{gl}}.
$ 
Note that the projector is locally facet-by-facet defined, 
and the transformation matrix $\mathsf{P}$ is sparse.
For the numerical realization $\mathcal{A}^{-1}$, we 
again use hypre's BoomerAMG.
\section{Semidiscrete a priori error analysis}
\label{sec:err}
In this section, we present an a prior error analysis for the semidiscrete
scheme \eqref{semi}. To simplify notation, we write
\begin{align*}
	A\lesssim B
\end{align*}
to indicate that there exists a constant $C$, independent of mesh size $h$,
material parameters $\rho^{f/s}$, $\mu^{f/s}$, $\lambda^s$, and the numerical solution, such that $A\le CB$.

We denote the following (semi)norms:
\begin{subequations}
	\label{norms}
	\begin{align}
		\|
  (\bld v, \widehat{\bld v})
    \|_{i,h}
		&:=\;
\sum_{K\in\Th^i} 
\left(\|\bld D(\bld v_h)\|_K^2 
+\frac{\alpha k^2}{h}
\|\Pi_h(\mathsf{tang}(\bld v_h-\widehat{\bld v}_h))\|_{\partial K}^2
\right),\\
		\|
  (\bld v, \widehat{\bld v})
    \|_{i,*,h} 
		&:=\;
		\left(\|
  (\bld v, \widehat{\bld v})
    \|_{i,h}^2+\sum_{K\in\Th^i}h\|\bld D(\bld{v})\|_{\partial K}^2\right)^{1/2},\\
		|\!|\!|\{\bld{v},
    \bld \xi^s, \widehat{\bld \xi^s}\}
    |\!|\!|_h 
		&:=\;
		\left(\|\rho^{1/2}\bld{v}\|^2+2\mu^s\|
  (\bld \xi^s, \widehat{\bld \xi^s})
    \|_{s,h}^2+\lambda^s\|\nabla\cdot\bld{\xi}^s\|^2_s\right)^{1/2},
	\end{align}
\end{subequations}
for $i\in\{f,s\}$ 
and $(\bld v,\widehat{\bld v}, \bld \xi^s, 
\widehat{\bld \xi^s})\in 
\bld V_{h,0}^k\times \widehat{\bld V}_{h,0}^{k-1}
\times\bld V_{h,0}^{k,s}\times \widehat{\bld V}_{h,0}^{k-1,s} 
$, where we denote $\|\cdot\|$ as the $L^2$-norm on $\Omega$, $\|\cdot\|_i$ as the $L^2$-norm on $\Omega^i$.
The inequality \eqref{coercivity} implies the 
coercivity of the bilinear form $A_h^{i}$  with respect to
the norm $\|\cdot\|_{i,h}$.
We also have the following boundedness of the operator $A_h^{i}$:
\begin{align}
	\label{boundNorm}
	A_h^{i}(
  (\bld v, \widehat{\bld v}),
(\bld w_h, \widehat{\bld w}_h)
  )
	&\lesssim
	\|
  (\bld v, \widehat{\bld v})
  \|_{i,*,h}
	\|
  (\bld w_h, \widehat{\bld w}_h)
  \|_{i,h}, 
\end{align}
for all
 $ (\bld v, \widehat {\bld v})\in
 \underline{\bld V}^i+\left(
 \bld V_{h,0}^{k,i}\times \widehat{\bld V}_{h,0}^{k-1, i}\right)$ and 
 $ (\bld w_h, \widehat {\bld w}_h)\in
  \bld V_{h,0}^{k,i}\times \widehat{\bld V}_{h,0}^{k-1, i}$, 
where 
\[
  \underline{\bld V}^i:=\{(\bld v, \bld v|_{\Eh^i}):\;\;
  \bld v|_K\in H^2(K), \;\;\forall K\in\Th^i\}.
\] 

We use the  classical Brezzi-Douglas-Marini (BDM) interpolator
$\Pi_{BDM}$\cite[Proposition 2.3.2]{boffi2013mixed} to project $\bld{u}$ and
$\bld{\eta}^s$ onto the finite element spaces $\bld{V}^k_{h,0}$ and
$\bld{V}^{k,s}_{h,0}$. We denote $\Pi_Q$ as the $L^2$-projection onto the finite
element space $Q_h^{k-1}$. Note that due to the commuting projection property, we have:
\begin{subequations}
	\label{commut}
\begin{alignat}{3}
	(\nabla\cdot\Pi_{BDM}\bld{\eta^s},q_h^s)_s
	&=
	(\nabla\cdot\bld{\eta^s},q_h^s)_s, \quad&&\forall q_h^s\in Q_h^{k-1,s},
	\\
	\label{divFreeProj}
  \nabla\cdot\Pi_{BDM}\bld{u}^f|_{\Omega^f}
	&=
  \Pi_{Q}(\nabla\cdot\bld{u}^f)|_{\Omega^f}=0.
\end{alignat}
\end{subequations}

The following standard approximation property of the BDM projector
$\Pi_{BDM}$ and the $L^2$-projector $\Pi_h$ onto 
$\widehat{\bld V}_{h,0}^{k-1}$ is well-known; see \cite[Proposition 2.3.8]{Lehrenfeld10}.
\begin{lemma}
	\label{lemma:approx}
  Let $\bld u\in [H^1(\Omega)]^d\cap [H^{k+1}(\Th)]^d$.
  Then the following estimates hold:
		\begin{align}
		\label{approxErr}
    \|(\bld{u}-\Pi_{BDM}\bld u,\bld{u}|_{\Eh^i}-\Pi_h\bld{u})\|_{i,*,h}^2
		&\lesssim
    h^{2k}\sum_{K\in\Th^i}\|\bld{ u}\|_{H^{k+1}(K)}^2,
		\end{align}
    for $i\in\{f,s\}$.
\end{lemma}

To further simplify notation, we denote:
\begin{alignat*}{3}
		&\underline{\bld{\delta}_{\bld{u}}}
		&&:=
		(\bld{\delta}_{\bld{u}},\bld{\delta}_{\widehat{\bld{u}}})
		&&:=
		(\bld{u}-\Pi_{BDM}\bld{u},\bld{u}|_{\Eh}-\Pi_h\bld{u}),
		\\
		&\underline{\bld{\delta}_{\bld{\eta}^s}}
		&&:=
		(\bld{\delta}_{\bld{\eta}^s},\bld{\delta}_{\widehat{\bld{\eta}}^s})
		&&:=
		(\bld{\eta}^s-\Pi_{BDM}\bld{\eta}^s,\bld{\eta}^s|_{\Eh^s}-\Pi_h\bld{\eta}^s),
		\\
		&\delta_{p^f}&&:=p^f-\Pi_Q p^f&&,
		\\
		&\underline{\bld{\varepsilon}_{\bld{u}}}
		&&:=
		(\bld{\varepsilon}_{\bld{u}},\bld{\varepsilon}_{\widehat{\bld{u}}})
		&&:=
		(\bld{u}_h-\Pi_{BDM}\bld{u},\widehat{\bld{u}}_h-\Pi_h\bld{u}),
		\\
		&\underline{\bld{\varepsilon}_{\bld{\eta}^s}}
		&&:=
		(\bld{\varepsilon}_{\bld{\eta}^s},\bld{\varepsilon}_{\widehat{\bld{\eta}}^s})
		&&:=
		(\bld{\eta}^s_h-\Pi_{BDM}\bld{\eta}^s,\widehat{\bld{\eta}}^s_h-\Pi_h\bld{\eta}^s),
		\\
		&\varepsilon_{p^f}&&:=p^f_h-\Pi_Q p^f&&.
\end{alignat*}
where 
$(\underline{\bld{u}_h},p^f_h,\underline{\bld{\eta}^s_h})
\in 
	\underline{\bld{V}_{h,0}^k}\times Q_{h,0}^{k-1,f}\times\underline{\bld{V}_{h,0}^{k,s}}$
 is the solution to the semi-discrete scheme \eqref{semi}, with 
the componound spaces denoted as
\[
  \underline{\bld{V}_{h,0}^k}:=\bld {V}_{h,0}^k\times \widehat{\bld
  V}_{h,0}^{k-1},\;\;
  \underline{\bld{V}_{h,0}^{k,s}}:=\bld {V}_{h,0}^{k,s}\times \widehat{\bld
  V}_{h,0}^{k-1,s}.
\]
\begin{lemma}[Error equations of the semi-discrete scheme \eqref{semi}]
We have
the following error equations for the semi-discrete scheme \eqref{semi} :
	\begin{subequations}
		\label{semiErrEqu}
		\begin{align}
			\label{semiErrEqu-1}
			(\rho\partial_t\bld{\varepsilon}_{\bld{u}},\bld{v}_h)
			+2\mu^f A_h^f\left(\underline{\bld{\varepsilon}_{\bld{u}}},\underline{\bld{v}_h}\right)
			&-(\bld{\varepsilon}_{p^f},\nabla\cdot\bld{v}_h)_f
			+2\mu^s A_h^s\left(\underline{\bld{\varepsilon}_{\bld{\eta}^s}},\underline{\bld{v}_h}\right)
			+\lambda^s(\nabla\cdot\bld{\varepsilon}_{\bld{\eta}^s},\nabla\cdot\bld{v}_h)_s
			\\
			&=\;(\rho\partial_t\bld{\delta}_{\bld{u}},\bld{v}_h)
			+2\mu^f A_h^f\left(\underline{\bld{\delta}_{\bld{u}}},\underline{\bld{v}_h}\right)
			+2\mu^s A_h^s\left(\underline{\bld{\delta}_{\bld{\eta}^s}},\underline{\bld{v}_h}\right),\nonumber
			\\
			\label{semiErrEqu-2}
			(\partial_t\bld{\varepsilon}_{\bld{\eta}^s},\bld{\xi}_h^s)_s
			&=\;(\bld{\varepsilon}_{\bld{u}^s},\bld{\xi}_h^s)_s,\\
			\label{semiErrEqu-3}
			\langle\partial_t\bld{\varepsilon}_{\widehat{\bld{\eta}}^s},\widehat{\bld{\xi}}_h^s\rangle_s
			&=\;\langle\bld{\varepsilon}_{\widehat{\bld{u}}^s},\widehat{\bld{\xi}}_h^s\rangle_s.
		\end{align}
	\end{subequations}
	for all $(\underline{\bld{v}_h},q_h^f,\underline{\bld{\xi}_h^s})\in\underline{\bld{V}_{h,0}^k}\times Q_{h,0}^{k-1,f}\times\underline{\bld{V}_{h,0}^{k,s}}$.
\end{lemma}
\begin{proof}
	By subtracting the semi-discrete scheme \eqref{semiX-1} from the consistency result \eqref{semi-1}, then adding and subtracting the above projectors, we can get the error equation \eqref{semiErrEqu-1}, where the commutative property of BDM interpolation \eqref{commut} is used. Then \eqref{semiErrEqu-2} and \eqref{semiErrEqu-3} can be easily derived since we have $\partial_t\Pi_{BDM}\bld{u}^s=\partial_t\Pi_{BDM}\bld{\eta}^s$, $\partial_t\Pi_h\bld{u}^s=\partial_t\Pi_h\bld{\eta}^s$. 
\end{proof}

Note that due to the same finite elements space of velocity and displacement
approximation in $\Omega^s$, the error equations \eqref{semiErrEqu-2} and
\eqref{semiErrEqu-3} actually imply that
$\bld{\varepsilon}_{\bld{u}^s}=\partial_t\bld{\varepsilon}_{\bld{\eta}^s},
\bld{\varepsilon}_{\widehat{\bld{u}}^s}=\partial_t\bld{\varepsilon}_{\widehat{\bld{\eta}}^s}$.
Now we are ready to present the main result in this section.
\begin{theorem}
	\label{theorem:semiErr}
  Let $(\underline{\bld{u}_h},p_h^f,\underline{\bld{\eta}_h^s})$ be the solution
  to semi-discrete scheme \eqref{semi} with initial data such that
  $\left(\underline{\bld{\varepsilon}_{\bld{u}}}(0),\varepsilon_{p^f}(0),\underline{\bld{\varepsilon}_{\bld{\eta}^s}}(0)\right)=\left(\underline{\bld{0}},0,\underline{\bld{0}}\right)$.
  Assume the solution $(\bld u, \bld \eta^s)$ to the model problem \eqref{model}
  is smooth.
  Then the following estimation holds for all $T>0$:
	\begin{equation}
		\label{semiErr}
    |\!|\!| 
    \{\bld{\varepsilon}_{\bld{u}}(T),\underline{\bld{\varepsilon}_{\bld{\eta}^s}}(T)\}
    |\!|\!|_h^2
		+
		\mu^f\int_{0}^{T}
		\|\underline{\bld{\varepsilon}_{\bld{u}}}\|_{f,h}^2\mathrm{dt}
		\lesssim\; 
		h^{2k}\left(\lXi_1+\lXi_2+\lXi_3\right),
	\end{equation}
	where
	\begin{align*}
		\lXi_1
		&:=
    T\int_{0}^{T}\left(\|\rho^{1/2}\partial_t\bld{u}\|_{H^{k}(\Omega)}^2
		+
		\mu^s
		\|\partial_t\bld{\eta}^s\|_{H^{k+1}(\Omega^s)}
		\right)\mathrm{dt},
		\\
		\lXi_2
		&:=
		\mu^f\int_{0}^{T}
		\|\bld{u}\|_{H^{k+1}(\Omega^f)}^2\mathrm{dt},
		\\
		\lXi_3
		&:=
		\mu^s
		\|\bld{\eta}^s\|_{L^{\infty}\left(H^{k+1}(\Omega^s)\right)}.
	\end{align*}
\end{theorem}
\begin{proof}
	Here we use the standard energy argument. Take $(\underline{\bld{v}_h},q_h^f)=(\underline{\bld{\varepsilon}_{\bld{u}}},\varepsilon_{p^f})$ in error equation \eqref{semiErrEqu-1} and plug in $\bld{\varepsilon}_{\bld{u}^s}=\partial_t\bld{\varepsilon}_{\bld{\eta}^s}, \bld{\varepsilon}_{\widehat{\bld{u}}^s}=\partial_t\bld{\varepsilon}_{\widehat{\bld{\eta}}^s}$, we get:
	\begin{align*}
		\frac{1}{2}\frac{\partial}{\partial_t}&\underbrace{\left(\|\rho^{1/2}\bld{\varepsilon}_{\bld{u}}\|^2
		+
		2\mu^s A_h^s\left(\underline{\bld{\varepsilon}_{\bld{\eta}^s}},\underline{\bld{\varepsilon}_{\bld{\eta}^s}}\right)
		+
		\lambda^s\|\nabla\cdot\bld{\varepsilon}_{\bld{\eta}^s}\|_s^2\right)}_{:=\mathcal{H}(t)}
		+2\mu^f
		A_h^f\left(\underline{\bld{\varepsilon}_{\bld{u}}},\underline{\bld{\varepsilon}_{\bld{u}}}\right)
		\\
		&=\;
		(\rho\partial_t\bld{\delta}_{\bld{u}},\bld{\varepsilon}_{\bld{u}})
		+
		2\mu^f A_h^f\left(\underline{\bld{\delta}_{\bld{u}}},\underline{\bld{\varepsilon}_{\bld{u}}}\right)
		+
		2\mu^s
		A_h^s\left(\underline{\bld{\delta}_{\bld{\eta}^s}},\partial_t\underline{\bld{\varepsilon}_{\bld{\eta}^s}}\right),\nonumber
	\end{align*}
	where we used the exactly divergence-free property of $\bld{u}^f$ and
  $\Pi_{BDM}\bld{u}^f$. By plugging in the right-hand side the chain rule for
  the time derivative
	\begin{align*}
		\frac{d}{d t}
		A_h^s\left(\underline{\bld{\delta}_{\bld{\eta}^s}},\underline{\bld{\varepsilon}_{\bld{\eta}^s}}\right)
		=
		A_h^s\left(\partial_t\underline{\bld{\delta}_{\bld{\eta}^s}},\underline{\bld{\varepsilon}_{\bld{\eta}^s}}\right)
		+
		A_h^s\left(\underline{\bld{\delta}_{\bld{\eta}^s}},\partial_t\underline{\bld{\varepsilon}_{\bld{\eta}^s}}\right),
	\end{align*}
	and then applying the Cauchy-Schwarz inequality and boundedness of $A_h^i$ \eqref{boundNorm}, we get:
	\begin{align*}
		\frac{1}{2}\frac{\partial}{\partial_t}\mathcal{H}(t)
		+
		2\mu^f
		A_h^f\left(\underline{\bld{\varepsilon}_{\bld{u}}},\underline{\bld{\varepsilon}_{\bld{u}}}\right)
		\lesssim&
		\|\rho^{1/2}\partial_t\bld{\delta}_{\bld{u}}\|\,
		\|\rho^{1/2}\bld{\varepsilon}_{\bld{u}}\|
		+
		2\mu^f\|\underline{\bld{\delta}_{\bld{u}}}\|_{f,*,h}
		\|\underline{\bld{\varepsilon}_{\bld{u}}}\|_{f,h}
		\\
		&+
		2\mu^s
		\|\partial_t\underline{\bld{\delta}_{\bld{\eta}^s}}\|_{s,*,h}
		\|\underline{\bld{\varepsilon}_{\bld{\eta}^s}}\|_{s,h}
		+
		2\mu^s
		\frac{d}{dt}
		A_h^s\left(\underline{\bld{\delta}_{\bld{\eta}^s}},\underline{\bld{\varepsilon}_{\bld{\eta}^s}}\right)
		\\
		\lesssim&
		\Theta^{1/2}|\!|\!|\{\bld{\varepsilon}_{\bld{u}},\underline{\bld{\varepsilon}_{\bld{\eta}^s}}\}|\!|\!|_h
		+
		2\mu^f
		\|\underline{\bld{\delta}_{\bld{u}}}\|_{f,*,h}
		\|\underline{\bld{\varepsilon}_{\bld{u}}}\|_{f,h}
		+
		2\mu^s
		\frac{\partial}{\partial_t}
		A_h^s\left(\underline{\bld{\delta}_{\bld{\eta}^s}},\underline{\bld{\varepsilon}_{\bld{\eta}^s}}\right),
	\end{align*}
	where $\Theta:=\left(\|\rho^{1/2}\partial_t\bld{\delta}_{\bld{u}}\|^2+	2\mu^s
	\|\partial_t\underline{\bld{\delta}_{\bld{\eta}^s}}\|_{s,*,h}^2\right)$. Integrate both sides over time from $t=0$ to $t=T$, combined with $\left(\underline{\bld{\varepsilon}_{\bld{u}}}(0),\varepsilon_{p^f}(0),\underline{\bld{\varepsilon}_{\bld{\eta}^s}}(0)\right)=\left(\underline{\bld{0}},0,\underline{\bld{0}}\right)$, we get:
	\begin{align*}
		\mathcal{H}(T)
		+
		\mu^f
		\int_{0}^{T}A_h^f\left(\underline{\bld{\varepsilon}_{\bld{u}}},\underline{\bld{\varepsilon}_{\bld{u}}}\right)\mathrm{dt}
		\lesssim&
		\int_{0}^{T}
		\Theta^{1/2}|\!|\!|\{\bld{\varepsilon}_{\bld{u}},\underline{\bld{\varepsilon}_{\bld{\eta}^s}}\}|\!|\!|_h\mathrm{dt}
		+
		\mu^f\int_{0}^{T}
		\|\underline{\bld{\delta}_{\bld{u}}}\|_{f,*,h}
		\|\underline{\bld{\varepsilon}_{\bld{u}}}\|_{f,h}\mathrm{dt}
		\\
		&+
		\mu^s
		A_h^s\left(\underline{\bld{\delta}_{\bld{\eta}^s}}(T),\underline{\bld{\varepsilon}_{\bld{\eta}^s}}(T)\right).
	\end{align*}
  Applying the coercivity and boundedness of $A_h^i$, and Young's inequality, we get:
	\begin{align*}
		|\!|\!|\{\bld{\varepsilon}_{\bld{u}}(T),\underline{\bld{\varepsilon}_{\bld{\eta}^s}}(T)\}|\!|\!|_h^2
		+
		\mu^f\int_{0}^{T}
		\|\underline{\bld{\varepsilon}_{\bld{u}}}\|_{f,h}^2\mathrm{dt}
		\lesssim&
		\int_{0}^{T}
		\Theta^{1/2}|\!|\!|\{\bld{\varepsilon}_{\bld{u}},\underline{\bld{\varepsilon}_{\bld{\eta}^s}}\}|\!|\!|_h\mathrm{dt}
		\\
		&+
		\mu^f\gamma_1\int_{0}^{T}
		\|\underline{\bld{\delta}_{\bld{u}}}\|_{f,*,h}^2\mathrm{dt}
		+
		\mu^s\gamma_2
		\|\underline{\bld{\delta}_{\bld{\eta}^s}}(T)\|_{s,*,h}^2
		\\
		&+
		\frac{\mu^f}{\gamma_1}\int_{0}^{T}
		\|\underline{\bld{\varepsilon}_{\bld{u}}}\|_{f,h}^2\mathrm{dt}
		+
		\frac{\mu^s}{\gamma_2}
		\|\underline{\bld{\varepsilon}_{\bld{\eta}^s}}(T)\|_{s,h}^2,
	\end{align*}
	for all $\gamma_1,\gamma_2 > 0$. The last two terms would be absorbed by the left-hand side when $\gamma_1$ and $\gamma_2$ are big enough. Then we have:
	\begin{align*}
		|\!|\!|\{\bld{\varepsilon}_{\bld{u}}(T),\underline{\bld{\varepsilon}_{\bld{\eta}^s}}(T)\}|\!|\!|_h^2
		+
		\mu^f\int_{0}^{T}
		\|\underline{\bld{\varepsilon}_{\bld{u}}}\|_{f,h}^2\mathrm{dt}
		\lesssim&
		\int_{0}^{T}
		\Theta^{1/2}|\|\{\bld{\varepsilon}_{\bld{u}},\underline{\bld{\varepsilon}_{\bld{\eta}^s}}\}\||_h\mathrm{dt}
		\\
		&+
		\mu^f\int_{0}^{T}
		\|\underline{\bld{\delta}_{\bld{u}}}\|_{f,*,h}^2\mathrm{dt}
		+
		\mu^s
		\|\underline{\bld{\delta}_{\bld{\eta}^s}}(T)\|_{s,*,h}^2.
	\end{align*}
    By applying the Gronwall-type inequality \cite[Propostion 3.1]{chabaud2012uniform} and the Cauchy-Schwarz inequality, we get:
    \begin{align*}
    	|\!|\!|\{\bld{\varepsilon}_{\bld{u}}(T),\underline{\bld{\varepsilon}_{\bld{\eta}^s}}(T)\}|\!|\!|_h^2
    	&+
    	\mu^f\int_{0}^{T}
    	\|\underline{\bld{\varepsilon}_{\bld{u}}}\|_{f,h}^2\mathrm{dt}
    	\\
    	&\le
    	\left(\frac{1}{2}\int_{0}^{T}\Theta^{1/2}\mathrm{dt}
    	+
    	\left(\max_{0\le\mathrm{t}\le T}
    	\left(
    	\mu^f\int_{0}^{t}
    	\|\underline{\bld{\delta}_{\bld{u}}}\|_{f,*,h}^2\mathrm{dt}
    	+
    	\mu^s
    	\|\underline{\bld{\delta}_{\bld{\eta}^s}}\|_{s,*,h}^2
    	\right)
    	\right)^{1/2}
    	\right)^2
    	\\
    	&\lesssim
    	T\int_{0}^{T}\Theta\mathrm{dt}
    	+
    	\mu^f\int_{0}^{T}
    	\|\underline{\bld{\delta}_{\bld{u}}}\|_{f,*,h}^2\mathrm{dt}
    	+
    	\mu^s
    	\max_{0\le\mathrm{t}\le T}
    	\|\underline{\bld{\delta}_{\bld{\eta}^s}}\|_{s,*,h}^2.
    \end{align*}
    Finally, the estimate \eqref{semiErr} is obtained by the above inequality
    and the approximation properties of the projectors in Lemma
    \ref{lemma:approx}.
\end{proof}

\begin{remark}[Robust velocity/displacement estimates]
  It is clear that the velocity and displacement error estimate  \eqref{semiErr}  is independent of the pressure approximation $p_h^f$ and 
  the lame parameter $\lambda^s$.
  Moreover, the error estimate \eqref{semiErr} is optimal in the energy norm 
  $|\!|\!|\cdot|\!|\!|_h$, which contains a discrete $H^1$-norm on $\Omega^s$.
On the other hand,  we can only obtain a suboptimal convergence of order $\mathcal{O}(h^k)$ for
  the $L^2$-norm of the velocity approximation from \eqref{semiErr}.
However,  our numerical results in the next section indicate that the velocity $L^2$-norm
  seems to be optimal. The proof of the optimality of the velocity $L^2$-norm is our
  future work.
\end{remark}
\section{Numerical results}
\label{sec:num}
In this section, we 
present {three} numerical examples for the 
model problem \eqref{model} in two k{and three} dimensions.
The first example 
uses a manufactured solution to verify the accuracy of the proposed 
 monolithic divergence-conforming HDG
 schemes \eqref{full} and \eqref{fullB}
and the robustness of the preconditioner \eqref{prec} 
with respect to mesh size, time step
size, and material parameters.
The second example is a classical benchmark problem typically used to
validate FSI solvers \cite{Nobile01,Bukac14x}. 
{The third example is a 3D test case simulating the propagation of pressure pulse through a  straight cylinder pipe.}
The NGSolve software \cite{Schoberl16} is used for the simulations.

\subsection{Example 1: The method of manufactured solutions}
We consider  a rectangular fluid
domain, $\Omega^f = (0, 1)\times (-1, 0)$, and a rectangular solid domain,
$\Omega^s = (0, 1) \times (0, {0.5})$, connected by an
interface, $\Gamma = \{(x, y): \; x \in (0, 1), y = 0\}$.
We choose the volume and interface source terms such that the exact solutions
are given as follows:
\begin{align*}
\bld u^f=\bld u^s =&\;\left( \sin(2\pi x)^2\sin(\frac{8}{3}\pi(y+1))\sin(2t),
-1.5\sin(4\pi x)\sin(\frac43\pi (y+1))^2\sin(2t)\right),\\
  p^f=&\; 
\sin(2\pi x)\sin(2\pi y)\sin(t),\\
\bld \eta^s =&\;\left( \sin(2\pi x)^2\sin(\frac{8}{3}\pi(y+1))\sin(t)^2,
-1.5\sin(4\pi x)\sin(\frac43\pi (y+1))^2\sin(t)^2\right).
\end{align*}
We use a homogeneous Dirichlet boundary conditions \eqref{bcbc} on the
exterior boundaries.
For the material parameters, 
we take the fluid density and viscosity to be one ($\rho^f=\mu^f=1$), 
and vary the structure density and Lam\'e parameters in large parameter ranges:
$$
\rho^s \in\{10^{-3}, 1, 10^{3}\},
\mu^s = \delta_1\, \rho^s, \text{with } \delta_1\in\{0.1, 1, 10\},
\text{and } 
\lambda^s = \delta_2\, \mu^s, \text{with } \delta_2\in\{1, 10^4\}.
$$
Here $\delta_2=1$ corresponds to a compressible structure, while
$\delta_2=10^4$ corresponds to a nearly incompressible structure.

We run simulations on a sequence of uniform unstructured 
triangular meshes with mesh size $h=\frac{1}{10\times 2^j}
$ for $j=0,1,2,3$. 
We take the polynomial degree to be either $k=1$ or $k=2$.
We use the (second-order) Crank-Nicolson temporal
discretization \eqref{full} for $k=1$,  and 
the (third-order) BDF3 temporal discretization \eqref{fullB}, and
take a uniform time step size $\delta t = h$. 
To start the BDF3 scheme, we compute 
$(\bld u_h^m, \bld \eta_h^{s,m}, \widehat{\bld \eta}_h^{s,m})$
by interpolating the exact solution at time $t_m=m\,\delta t$, $m=0,1,2$.
The preconditioned MinRes solver with the preconditioner
\eqref{prec} with AMG blocks  \eqref{s-pre} and \eqref{a-pre} 
is used to solve the linear system in each
time step, for which we start with {\it zero} initial 
guess and stop until the residual norm is decreased by a factor of
$10^{-8}$.

The $L^2$-errors in the velocity approximation 
$\|\bld u-\bld u_h\|_{\Omega}$ at the final time $T=0.3$ are documented in 
Table \ref{table:m1}--\ref{table:m2} for various parameter choices.
It is clear to observe that our fully discrete scheme provide an optimal 
velocity approximation of order 2 for polynomial degree $k=1$ with Crank-Nicolson time
stepping, and of order 3 for $k=2$ with BDF3 time stepping. 
Moreover, we observe that our fully discrete scheme is robust with respect to 
large density variations and large Lam\'e parameter variations since the errors for different parameters in each row of Table
\ref{table:m1}--\ref{table:m2} are similar.

The average numbers of iterations needed for the convergence of the
preconditioned MinRes solver are recorded in 
Table \ref{table:m3}--\ref{table:m4}.
We observe for polynomial degree $k=1$, we roughly need about $150$
iterations
to converge for the compressible structure case in Table \ref{table:m3}
and about $116$ iterations for the nearly incompressible structure case in 
Table \ref{table:m4}. Also, the preconditioner is fairly robust with
respect to the mesh size (and time step size), and parameter variations in 
$\rho^s$ and $\mu^s$.
Similar results are observed for the $k=2$ case, which needs
roughly about $285$ iterations to converge for the compressible case in Table
\ref{table:m3} and about 
$210$ iterations for the nearly incompressible case. However, it is also clear that
the preconditioner is not robust with respect to polynomial degree $k$.
We finally point out that the $k$-dependency on the iteration counts is due to the 
auxiliary space velocity preconditioner \eqref{a-pre} since if we replace 
$\widehat{\mathsf{iA}}$ by the exact inverse  $\mathsf{A}^{-1}$, the
iteration counts  are then observed 
to be quite insensitive to the polynomial degree:
about  30--40 iterations are needed  in the compressible cases, and about 
20--30 iterations in the nearly incompressible cases
for polynomial degree $k=1,2,3,4$.
This is expected as the polynomial degree in the pressure block 
is kept to be 0 regardless of the velocity polynomial degree $k$ in the
global linear system due to static condensation; see Remark \ref{rk:condense}.

\begin{table}[ht]
\centering 
\resizebox{0.9\columnwidth}{!}
{%
\begin{tabular}{cc| ccc|ccc|ccc}
\toprule
  & & \multicolumn{3}{c}{$\rho^s=10^{-3}$}
          & \multicolumn{3}{|c}{$\rho^s=1$}
          & \multicolumn{3}{|c}{$\rho^s=10^3$}
           \\
          & 
          & $\delta_1=0.1$ &$\delta_1=1$ &$\delta_1=10$ 
          & $\delta_1=0.1$ &$\delta_1=1$ &$\delta_1=10$ 
          & $\delta_1=0.1$ &$\delta_1=1$ &$\delta_1=10$ 
\tabularnewline
 \midrule
$k$& $1/h$  
   & error & error& error
   & error & error& error
   & error & error& error
\tabularnewline
\midrule
\multirow{4}{2mm}{1} 
   & 10& 3.492e-02& 3.420e-02& 5.624e-02& 3.489e-02& 3.408e-02& 5.350e-02&3.460e-02&  3.496e-02& 4.566e-02\\   
  &  20& 8.409e-03& 8.362e-03& 1.145e-02& 8.400e-03& 8.345e-03&1.085e-02&8.312e-03&  8.531e-03& 1.454e-02  \\ 
   &  40& 2.052e-03& 2.074e-03& 3.021e-03& 2.051e-03& 2.068e-03&2.777e-03&2.033e-03&  2.102e-03& 3.247e-03   \\
   &  80& 5.063e-04& 5.125e-04& 9.015e-04& 5.059e-04& 5.113e-04&8.126e-04&4.974e-04&  5.260e-04& 9.448e-04   \\
\midrule
\multicolumn{2}{c}{rate} 
   & 
2.04 &   2.02 &   1.98 &   2.04  &   2.02  &  2.01 &   2.04 &   2.02 &   1.89    
\\ [1.5ex]
 \midrule
\multirow{4}{2mm}{2} 
    & 10&4.124e-03& 4.273e-03& 4.331e-03&4.120e-03& 4.260e-03& 4.288e-03&4.116e-03& 4.279e-03& 4.339e-03 \\   
   &  20&5.151e-04& 5.298e-04& 5.262e-04&5.148e-04& 5.283e-04& 5.239e-04&5.136e-04& 5.442e-04& 5.259e-04  \\ 
   &  40&6.267e-05& 6.549e-05& 6.476e-05&6.265e-05& 6.548e-05& 6.564e-05&6.269e-05& 7.180e-05& 6.390e-05   \\
   &  80&7.733e-06& 8.028e-06& 7.712e-06&7.732e-06& 8.039e-06& 7.819e-06&7.738e-06& 9.032e-06& 8.915e-06   \\
\midrule
\multicolumn{2}{c}{rate} 
   & 3.02&    3.02&    3.04
   & 3.02&    3.02&    3.03
   & 3.02&    2.96&    2.98\\ 
\bottomrule
\end{tabular}}
\vspace{2ex}
\caption{\it \textbf{Example 1:} History of convergence of the $L^2$-velocity
errors. Compressible structure ($\delta_2=1$).} 
\label{table:m1} 
\end{table}

\begin{table}[ht]
\centering 
\resizebox{0.9\columnwidth}{!}
{%
\begin{tabular}{cc| ccc|ccc|ccc}
\toprule
  & & \multicolumn{3}{c}{$\rho^s=10^{-3}$}
          & \multicolumn{3}{|c}{$\rho^s=1$}
          & \multicolumn{3}{|c}{$\rho^s=10^3$}
           \\
          & 
          & $\delta_1=0.1$ &$\delta_1=1$ &$\delta_1=10$ 
          & $\delta_1=0.1$ &$\delta_1=1$ &$\delta_1=10$ 
          & $\delta_1=0.1$ &$\delta_1=1$ &$\delta_1=10$ 
\tabularnewline
 \midrule
$k$& $1/h$  
   & error & error& error
   & error & error& error
   & error & error& error
\tabularnewline
\midrule
\multirow{4}{2mm}{1} 
    & 10& 3.388e-02& 3.304e-02& 5.068e-02 &3.388e-02 &3.351e-02 &4.935e-02&3.382e-02& 3.478e-02& 4.694e-02 \\   
   &  20& 8.211e-03& 8.094e-03& 1.006e-02 &8.201e-03 &8.227e-03 &9.373e-03&8.088e-03& 8.400e-03& 1.248e-02  \\ 
   &  40& 2.004e-03& 1.998e-03& 2.136e-03 &2.002e-03 &2.022e-03 &2.053e-03&1.980e-03& 2.072e-03& 3.486e-03   \\
   &  80& 4.949e-04& 4.942e-04& 8.038e-04 &4.943e-04 &4.999e-04 &7.259e-04&4.861e-04& 5.180e-04& 9.316e-04   \\
\midrule
\multicolumn{2}{c}{rate} 
   & 
2.03  &  2.02  &  2.02 &2.03   & 2.02  &  2.05  & 2.04 &  2.02  & 1.88 
\\ [1.5ex]
\midrule
\multirow{4}{2mm}{2} 
    & 10&4.195e-03& 4.354e-03& 4.406e-03 &4.164e-03& 4.298e-03& 4.307e-03&4.133e-03& 4.311e-03& 4.374e-03   \\   
   &  20&5.200e-04& 5.296e-04& 5.221e-04 &5.181e-04& 5.237e-04& 5.272e-04&5.155e-04& 5.457e-04& 5.253e-04    \\ 
   &  40&6.258e-05& 6.421e-05& 6.430e-05 &6.245e-05& 6.419e-05& 6.548e-05&6.267e-05& 7.149e-05& 6.446e-05     \\
   &  80&7.697e-06& 7.845e-06& 7.708e-06 &7.691e-06& 7.886e-06& 7.860e-06&7.727e-06& 8.962e-06& 8.890e-06     \\
\midrule
\multicolumn{2}{c}{rate} 
    &3.03&    3.04&    3.05
    &3.03&    3.03&    3.03 &
3.02&    2.97&    2.99
   \\ 
\bottomrule
\end{tabular}}
\vspace{2ex}
\caption{\it \textbf{Example 1:} History of convergence of the $L^2$-velocity errors.
Nearly incompressible structure ($\delta_2=10^4$).} 
\label{table:m2} 
\end{table}

\begin{table}[ht]
\centering 
\resizebox{0.9\columnwidth}{!}
{%
\begin{tabular}{cc| ccc|ccc|ccc}
\toprule
  & & \multicolumn{3}{c}{$\rho^s=10^{-3}$}
          & \multicolumn{3}{|c}{$\rho^s=1$}
          & \multicolumn{3}{|c}{$\rho^s=10^3$}
           \\
          & 
          & $\delta_1=0.1$ &$\delta_1=1$ &$\delta_1=10$ 
          & $\delta_1=0.1$ &$\delta_1=1$ &$\delta_1=10$ 
          & $\delta_1=0.1$ &$\delta_1=1$ &$\delta_1=10$ 
\tabularnewline
 \midrule
$k$& $1/h$  
   & iter & iter & iter
   & iter & iter & iter
   & iter & iter & iter
\tabularnewline
\midrule
\multirow{4}{2mm}{1} 
    & 10& 136& 142& 122&137& 141& 122&154& 151& 128\\   
   &  20& 135& 146& 131&136& 148& 132&150& 157& 140 \\ 
   &  40& 148& 158& 152&145& 160& 153&149& 155& 154  \\
   &  80& 161& 174& 177&159& 180& 181&158& 169& 175  \\
 \midrule
\multirow{4}{2mm}{2} 
    & 10&281& 290& 250 &283& 291& 243& 289& 288& 238 \\   
   &  20&283& 302& 269 &284& 302& 263& 285& 288& 255 \\ 
   &  40&294& 313& 297 &293& 313& 291& 281& 287& 274 \\
   &  80&291& 310& 307 &288& 307& 303& 272& 279& 264 \\
\bottomrule
\end{tabular}}
\vspace{2ex}  
\caption{\it \textbf{Example 1:} 
    Average iteration counts for the preconditioned MinRes solver.
  Compressible structure ($\delta_2=1$).} 
\label{table:m3} 
\end{table}

\begin{table}[ht]
\centering 
\resizebox{0.9\columnwidth}{!}
{%
\begin{tabular}{cc| ccc|ccc|ccc}
\toprule
  & & \multicolumn{3}{c}{$\rho^s=10^{-3}$}
          & \multicolumn{3}{|c}{$\rho^s=1$}
          & \multicolumn{3}{|c}{$\rho^s=10^3$}
           \\
          & 
          & $\delta_1=0.1$ &$\delta_1=1$ &$\delta_1=10$ 
          & $\delta_1=0.1$ &$\delta_1=1$ &$\delta_1=10$ 
          & $\delta_1=0.1$ &$\delta_1=1$ &$\delta_1=10$ 
\tabularnewline
 \midrule
$k$& $1/h$  
   & iter & iter & iter
   & iter & iter & iter
   & iter & iter & iter
\tabularnewline
\midrule
\multirow{4}{2mm}{1} 
    & 10&115& 103& 108&115& 105& 111&134& 106& 109    \\   
   &  20&115& 101& 106&116& 103& 108&130& 108& 112     \\ 
   &  40&125& 108& 114&124& 110& 111&130& 105& 107      \\
   &  80&138& 117& 123&138& 123& 127&134& 113& 121      \\
\midrule
\multirow{4}{2mm}{2} 
    & 10&231& 200& 199&231& 199& 188&239& 195& 210 \\   
   &  20&228& 198& 200&228& 199& 188&239& 186& 205  \\ 
   &  40&232& 206& 213&232& 205& 202&237& 171& 200   \\
   &  80&229& 208& 216&226& 206& 208&231& 176& 186   \\
\bottomrule
\end{tabular}}
\vspace{2ex}
\caption{\it\textbf{Example 1:} 
    Average iteration counts for the preconditioned MinRes solver.
Nearly incompressible structure ($\delta_2=10^4$).} 
\label{table:m4} 
\end{table}

\subsection{Example 2: a linear two-dimensional test case}
We consider a simplified linear version of the numerical experiment reported in
\cite{Nobile01,Bukac14x}. We use the similar set-up as in \cite{Bukac14x}. 
We consider a fluid domain, $\Omega^f=(0,6)\times (0,0.5)[\mathrm{cm}]^2$, and 
a structure domain, $\Omega^s=(0,6)\times (0.5,0.6) [\mathrm{cm}]^2$, connected by an
interface $\Gamma =\{ (x,y):\;x\in(0,6), y = 0.5\}$. 
We consider the FSI problem \eqref{f-eq}--\eqref{interface} with
$\bld f^f=\bld f^s=0$, where we add
a linear {\it spring} term, $\beta^s \bld \eta^s$ to the 
first equation in \eqref{s-eq}:
$$
\rho^s \partial_t\bld u^s+\beta^s\bld \eta^s
-\nabla\cdot \bld \sigma^s(\bld \eta^s) = 0.
$$
The material parameters are given as follows:
$
\rho^s = 1.1[\mathrm{g/cm^3}],$ $
\mu^s = 0.575\times 10^6[\mathrm{dye/cm^2}],$ $
\beta^s = 4\times 10^6[\mathrm{dye/cm^4}],$ $
\lambda^s = 1.7\times 10^6[\mathrm{dye/cm^2}],$ $
\rho^f = 1[\mathsf{g/cm^3}],$ $
\mu^f = 0.035[\mathsf{g/(cm\cdot s)}]
$,
which are within physiologically realistic
values of blood flow in compliant arteries.
The flow is initially at rest, and we take the following boundary conditions
which model a pressure driven flow:
\begin{alignat*}{3}
  (\bld \sigma^f\bld n)\cdot\bld n=&\; -p_{in}(t), \;\;\;&&
  \mathsf{tang}(\bld u^f)=\; 0     \quad &&\text{on
  } \Gamma^f_{in}:=\{(x,y): x=0, y\in(0,0.5)\},\\
  (\bld \sigma^f\bld n)\cdot \bld n=&\; 0,\;&&
  \mathsf{tang}(\bld u^f)=\; 0\quad &&\text{on
  } \Gamma^f_{out}:=\{(x,y): x=6, y\in(0,0.5)\},\\
  \mathsf{tang}(\bld \sigma^f\bld n)=&\;0,\;&&
  \bld u^f\cdot\bld n=\; 0 \quad&& \text{on } \Gamma^f_{bot}:=\{(x,y):
  x\in(0,6), y = 0\},\\
  \bld \eta^s\cdot\bld n=&\; 0, \;\;\;&&
  \mathsf{tang}(\bld \eta^s)=\; 0     \quad &&\text{on
  } \Gamma^s_{in/out}:=\{(x,y): x\in\{0,6\}, y\in(0.5,0.6)\},\\
  (\bld \sigma^s\bld n)\cdot\bld n=&\;0,\;&&
  \mathsf{tang}(\bld \eta^s)=\; 0 
  \quad&& \text{on } \Gamma^s_{top}:=\{(x,y):
  x\in(0,6), y = 0.6\},
\end{alignat*}
where the  time-dependent
pressure boundary source term at the inlet $\Gamma_{in}^f$ is given as
follows:
\begin{align*}
  p_{in}(t)=\left\{
    \begin{tabular}{ll}
      $\frac{p_{\max}}{2}\left(1-\cos(\frac{2\pi t}{t_{\max}})\right)$,
      & if $t \le t_{\max}$, \\
      $0$, & if $t>t_{\max}$,
  \end{tabular}
  \right.
\end{align*}
where $t_{\max} = 0.03 [\mathrm{s}]$ and 
$p_{\max} = 1.333\times 10^4 [\mathrm{dyne/cm^2}]$.
The final time of the simulation is $T=1.2\times 10^{-2} [\mathrm{s}]$.

In this example, we use the divergence-conforming HDG scheme with
Crank-Nicolson time stepping \eqref{full}. 
The additional spring term $\beta^s \bld \eta^s$ 
in the structure equation does not alter the 
form of the resulting global linear system. 
Hence we still apply the preconditioned
MinRes solver using the preconditioner \eqref{prec} 
with AMG blocks \eqref{a-pre} and \eqref{s-pre}.
Due to different boundary conditions, we shall add the boundary
contribution
$$
\sum_{F\in \Gamma_{in}^f\cup \Gamma_{out}^f\cup \Gamma_{top}^s}
\int_F \frac{1}{\rho\,h}p_hq_h\,\mathrm{ds}
$$
to the bilinear form \eqref{neumann} associated with the matrix 
$\mathsf{N}_h^{\rho,\gamma}$ in the pressure block \eqref{s-pre}, 
and take the following continuous linear velocity auxiliary finite element space 
with the modified boundary conditions
{\small
$$
\bld V_h^{cg}:= \{\bld v\in \bld H^{1}(\Omega):\;\;
\bld v|_K\in [\mathcal{P}^1(K)]^d, \;\forall K\in \Th, 
\;\bld v|_{\Gamma_{in/out}^s}=0,
\;\bld v\cdot \bld n|_{\Gamma_{bot}^f}=0,
\;\mathsf{tang}(\bld v)|_{\Gamma_{in/out}^f\cup \Gamma^s_{top}}=0
\}
$$ }
in the velocity block \eqref{a-pre}.

For the discretization parameters, we consider
polynomial degree $k\in\{1,2,4\}$, 
a uniform unstructured triangular mesh with mesh size $h\in \{0.1,
0.05,0.025\}$,
and a uniform time step size $\delta t \in\{10^{-4}, 0.25\times 10^{-4}\}$.
For all the numerical simulations, 
we stop the MinRes iteration when residual norm is decreased by a factor of
$tol=10^{-6}$. The average number of MinRes iterations for different
discretization parameters are documented in Table \ref{table:mm}.
From Table \ref{table:mm}, we observe that 
\begin{itemize}
  \item [(a)] 
for the same polynomial degree $k$ and mesh size $h$, a smaller time step size $\delta t$ leads to a smaller number of MinRes iterations.
  \item [(b)] 
for the same mesh size $h$ and time step size $\delta t$, a larger
polynomial degree $k$ leads to a larger number of  MinRes iterations, with 
the number of iterations roughly doubled from $k=1$ to $k=4$.
\item [(c)]
for the same time step size $\delta t$ and 
polynomial degree $k$, the number of  MinRes iterations roughly stays in
the same level as mesh size $h$ decreases.
\end{itemize}
We also mention that the MinRes iterations in Table \ref{table:mm} are
smaller than those in Table \ref{table:m3}--\ref{table:m4} in Example 1,
which is partially due to the fact that we used a larger stopping tolerance
$tol = 10^{-6}$ here. 

\begin{table}[ht]
\centering 
{%
\begin{tabular}{cc| ccc|cccccc}
\toprule
  & & \multicolumn{3}{c}{$\delta t=10^{-4}$}
  & \multicolumn{3}{|c}{$\delta t=0.25\times 10^{-4}$}
           \\
          & 
          & $k=1$ &$k=2$ &$k=4$ 
          & $k=1$ &$k=2$ &$k=4$ 
\tabularnewline
 \midrule
   & $1/h$  
   & iter & iter & iter
   & iter & iter & iter
\tabularnewline
\midrule
    & 10&  76& 133& 213&  59& 79&143 \\   
   &  20&  89& 106& 158&  60& 83&134 \\ 
   &  40&  89& 115& 167&  72& 98&150  \\
\bottomrule
\end{tabular}}
\vspace{2ex}
  \caption{\it \textbf{Example 2:} 
    Average iteration counts for the preconditioned MinRes solver.} 
\label{table:mm} 
\end{table}


Finally, we plot in Figure \ref{fig:ex2} the flow rate, which is calculated
as two thirds of the horizontal velocity, 
and pressure at the bottom boundary $\Gamma^f_{bot}$, and the vertical 
displacement on the interface $\Gamma$ at final time $t=1.2\times
10^{-2}$ for 
$k=1$ with mesh size $h\in\{0.05,0.025\}$ and time step size $\delta t=10^{-4}$, 
$k=2$ with mesh size $h\in\{0.1,0.05\}$ and time step size
$\delta t=10^{-4}$, along with reference data for 
$k=4$ with mesh size $h = 0.025$
and time step size $\delta t=0.25\times 10^{-4}$.
We observe that both the results for $k=1$ and $k=2$ agrees well with the
reference data. We also observe that the result for 
$k=2$ on the coarse mesh with mesh size $h=0.1$ is more accurate that that
for $k=1$ on the medium
mesh with mesh size $h=0.05$, which indicates the benefits of using
a scheme with a higher order spatial discretization.
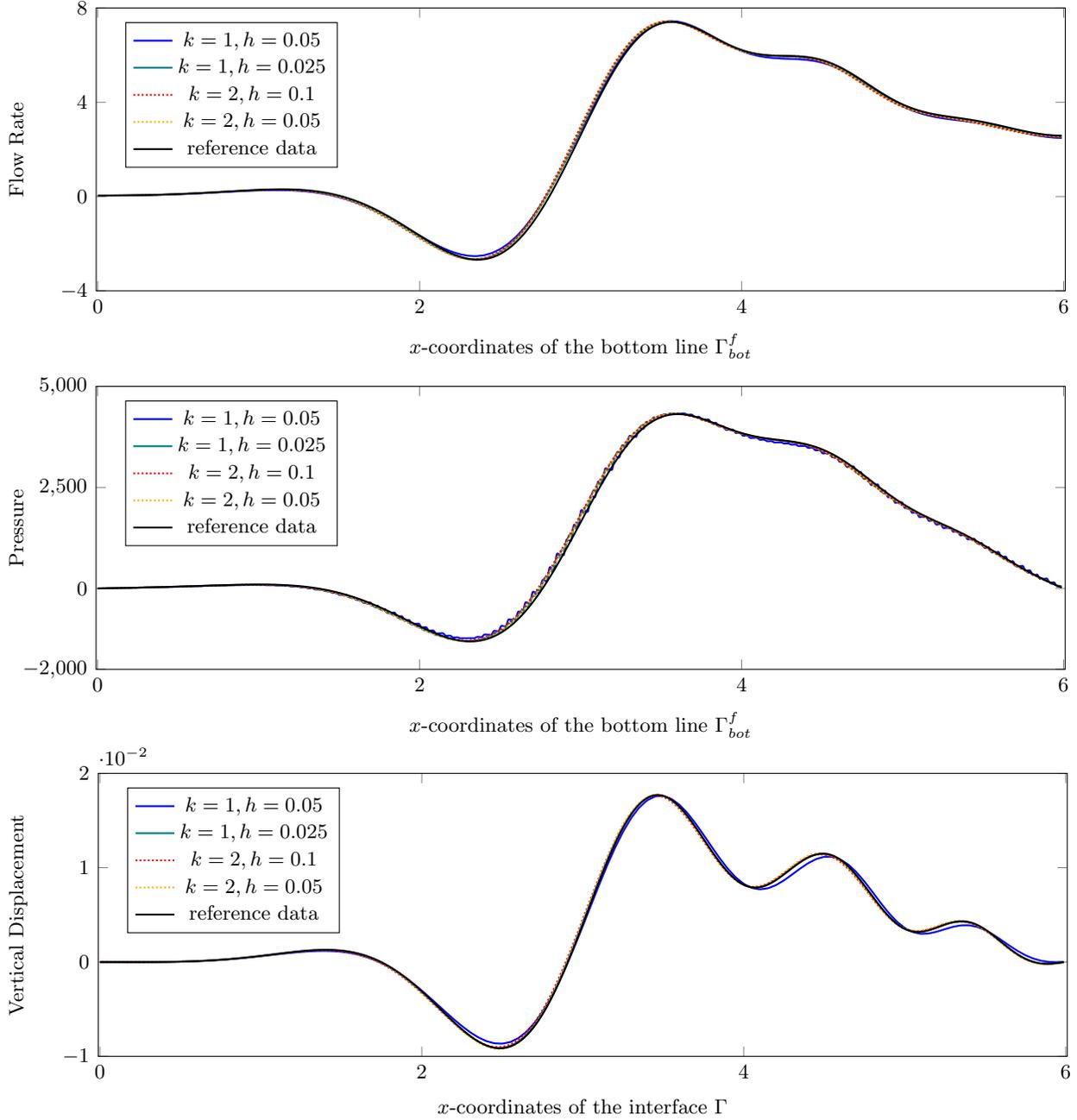
\begin{figure}
\begin{tikzpicture} 
     \tikzstyle{every node}=[font=\footnotesize]
\begin{axis}[
	width=\textwidth,
	height=0.26\textheight,
            xtick={0,2,4,6},
    ytick={-4,0,4, 8},
    xmin=-0.01,
    xmax=6.01,
    ymin=-4,
    ymax=8,
   	yticklabel style={/pgf/number format/fixed,/pgf/number format/precision=1},   	
   	every axis plot/.append style={line width=0.8pt, smooth},
    ylabel={Flow Rate},
    xlabel={$x$-coordinates of the bottom line  $\Gamma_{bot}^f$},
    legend style={at={(0.03,0.95)},anchor=north west}
   ]
\addplot[blue, solid] table[x index=0, y index=1]{U3.txt};
\addlegendentry{$k=1, h = 0.05$}
\addplot[green!50!blue] table[x index=0, y index=2]{U3.txt};
\addlegendentry{$k=1, h = 0.025$}
\addplot[red, densely dotted] table[x index=0, y index=3]{U3.txt};
\addlegendentry{$k=2, h = 0.1$}
\addplot[red!30!yellow, densely dotted] table[x
  index=0, y index=4]{U3.txt};
\addlegendentry{$k=2, h = 0.05$}
\addplot[black] table[x
  index=0, y index=5]{U3.txt};
\addlegendentry{reference data}  
\end{axis}
\end{tikzpicture}

\begin{tikzpicture} 
     \tikzstyle{every node}=[font=\footnotesize]
\begin{axis}[
	width=\textwidth,
	height=0.26\textheight,
            xtick={0,2,4,6},
    ytick={-2000, 0,2500, 5000},
    xmin=-0.01,
    xmax=6.01,
    ymin=-2000,
    ymax=5000,
   	yticklabel style={/pgf/number format/fixed,/pgf/number format/precision=1},   	
   	every axis plot/.append style={line width=0.8pt, smooth},
    legend style={at={(0.03,0.95)},anchor=north west},
    ylabel={Pressure},
    xlabel={$x$-coordinates of the bottom line $\Gamma_{bot}^f$},
   ]
\addplot[blue, solid] table[x index=0, y index=1]{P3.txt};
\addlegendentry{$k=1, h = 0.05$}
\addplot[green!50!blue] table[x index=0, y index=2]{P3.txt};
\addlegendentry{$k=1, h = 0.025$}
\addplot[red, densely dotted] table[x index=0, y index=3]{P3.txt};
\addlegendentry{$k=2, h = 0.1$}
\addplot[red!30!yellow, densely dotted] table[x
  index=0, y index=4]{P3.txt};
\addlegendentry{$k=2, h = 0.05$}
\addplot[black] table[x
  index=0, y index=5]{P3.txt};
\addlegendentry{reference data}  
\end{axis}
\end{tikzpicture}

\begin{tikzpicture} 
     \tikzstyle{every node}=[font=\footnotesize]
\begin{axis}[
	width=\textwidth,
	height=0.26\textheight,
            xtick={0,2,4,6},
    ytick={-0.01,0,0.01, 0.02},
    xmin=-0.01,
    xmax=6.01,
    ymin=-.01,
    ymax=0.02,
   	yticklabel style={/pgf/number format/fixed,/pgf/number format/precision=1},   	
   	every axis plot/.append style={line width=0.8pt, smooth},
    ylabel={Vertical Displacement},
    xlabel={$x$-coordinates of the interface $\Gamma$},
    legend style={at={(0.03,0.95)},anchor=north west}
   ]
\addplot[blue, solid] table[x index=0, y index=1]{D3.txt};
\addlegendentry{$k=1, h = 0.05$}
\addplot[green!50!blue] table[x index=0, y index=2]{D3.txt};
\addlegendentry{$k=1, h = 0.025$}
\addplot[red, densely dotted] table[x index=0, y index=3]{D3.txt};
\addlegendentry{$k=2, h = 0.1$}
\addplot[red!30!yellow, densely dotted] table[x
  index=0, y index=4]{D3.txt};
\addlegendentry{$k=2, h = 0.05$}
\addplot[black] table[x
  index=0, y index=5]{D3.txt};
\addlegendentry{reference data}  
  \node[above,font=\large\bfseries] at (current bounding box.north)  {Vertical displacement along 
 the interface $y=0.5$};
\end{axis}
\end{tikzpicture}
\caption{\it \textbf{Example 2:}  \it
  Numerical solutions of the scheme \eqref{full} with different 
  discretization parameters  at final time $t=1.2\times 10^{-2}[\mathsf{s}]$.
Top: flow rate $\frac23 \bld v_h[0]$ along 
bottom line $\Gamma_{bot}^f$; 
Middle: pressure along bottom line $\Gamma_{bot}^f$;
Bottom: vertical displacement $\bld\eta^s_h[1]$ along the interface
$\Gamma$.
Reference data is obtained with the HDG scheme \eqref{full} using polynomial degree $k=4$, mesh size $h=0.025$, 
and time step size  $\delta t=0.25\times 10^{-4}$.
All the other methods use the time step size $\delta t= 10^{-4}$.
}
  \label{fig:ex2}
\end{figure}

\subsection{Example 3: a linear three-dimensional test case on a straight
cylindrical pipe}
Now we consider a 3D example that simulates the propagation of
the pressure pulse on a straight cylinder (see \cite{Deparis06}).
The fluid domain is a straight cylinder of radius $0.5 [\mathsf{cm}]$ and length
$5 [\mathsf{cm}]$, 
$\Omega^f = \{(x,y,z): x\in(0,5),\, y^2+z^2<(0.5)^2\}$,
the structure domain has a thickness of $0.1 [\mathsf{cm}]$,
$\Omega^s = \{(x,y,z): x\in(0,5),\, (0.5)^2<y^2+z^2<(0.6)^2\}$, 
and the interface 
$\Gamma = \{(x,y,z): x\in(0,5),\, y^2+z^2=(0.5)^2\}$.
We use the same material parameters as in Example 2.
The flow is initially at rest, and we take the same boundary conditions as
in Example 2 with the exception that a pure Neumann boundary condition 
$\bld \sigma^s\bld n=0$ is applied on the exterior structure boundary 
$\Gamma^s_{ext}:=\{(x,y,z): x\in(0,5), y^2+z^2=0.6^2\}$.

We apply the scheme \eqref{full} with 
time step size $\delta t = 10^{-4}$.
For the spatial discretization parameters, we consider two cases:
$k=1$ on a fine mesh with mesh size $h=0.05$ (264,288 tetrahedra), and 
$k=2$ on a coarse mesh with mesh size $h=0.1$ (33,036 tetrahedra).
The fine mesh is illustrated in Figure \ref{fig:ex3}.
\begin{figure}
  \centering
  \includegraphics[width=0.95\textwidth]{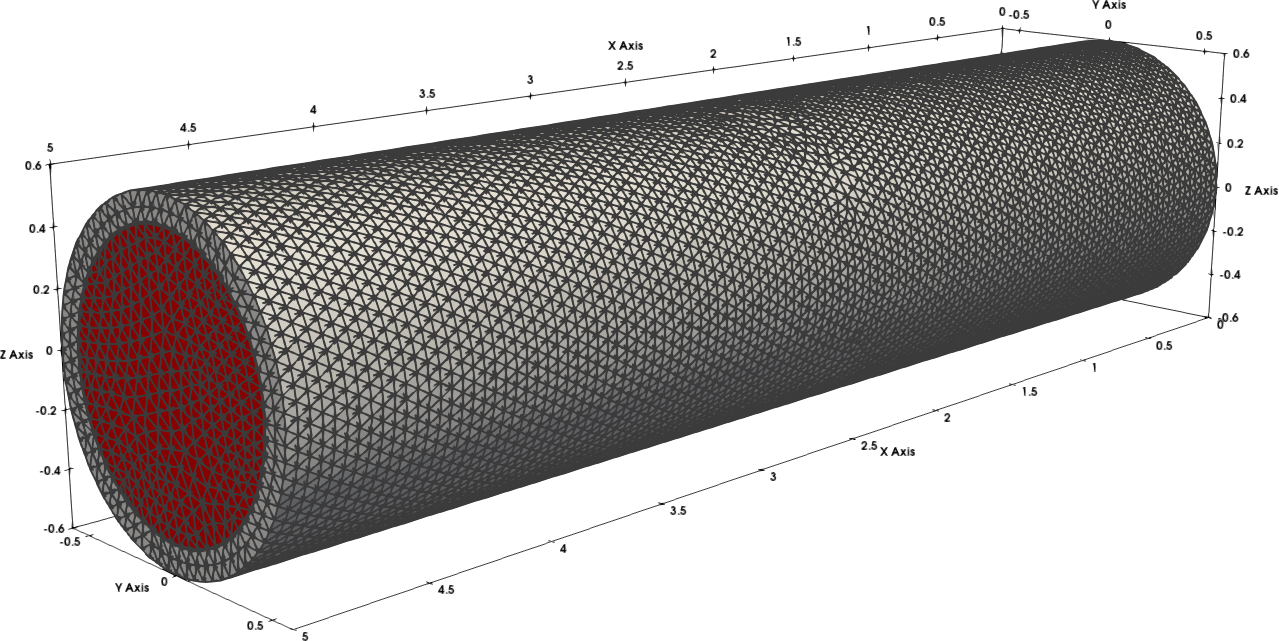}
  \caption{\it \textbf{Example 3:} the fine mesh with mesh size $h=0.05$.
  The red region is the fluid domain, and the gray region is the structure
domain.}
\label{fig:ex3}
\end{figure}
For the preconditioned MinRes linear system solver, 
we replace the point Gauss-Seidel smoother $\mathsf{R}$ 
in the  velocity preconditioner \eqref{a-pre} by a block Gauss-Seidel
smoother 
$\mathsf{R}^e$
based on edge blocks to further improve its efficiency. 
We stop the MinRes iteration when residual norm is decreased by a factor of
$10^{-6}$.
The average number of iterations for convergence for $k=1$ with $h=0.05$ is 
$60$ and that for $k=2$ with $h=0.1$ is $52$ when the edge-block
Gauss-Seidel smoother $\mathsf{R}^e$ is used in the velocity preconditioner
\eqref{a-pre}. If we instead use the 
point Gauss-Seidel smoother, the numbers would be 
$360$ for $k=1$ and $246$ for $k=2$.

Similar to Example 2, we 
plot in Figure \ref{fig:ex3Y} the flow rate, which is calculated
as two thirds of the horizontal velocity, 
and pressure at the center line $\{(x,0,0):\,x\in(0,5)\}$, and the
y-component of the displacement on the interface line 
$\{(x,0.5,0):\,x\in(0,5)\}$ at final time $t=1.2\times
10^{-2}$.
We find that the results for $k=1$ and $k=2$ agrees well with each
other.
\begin{figure}
\begin{tikzpicture} 
     \tikzstyle{every node}=[font=\footnotesize]
\begin{axis}[
	width=\textwidth,
	height=0.26\textheight,
            xtick={0,1,2,3,4,5},
    ytick={-4,0,4, 8,12},
    xmin=-0.01,
    xmax=5.01,
    ymin=-4.6,
    ymax=13,
   	yticklabel style={/pgf/number format/fixed,/pgf/number format/precision=1},   	
   	every axis plot/.append style={line width=0.8pt, smooth},
    ylabel={Flow Rate},
    xlabel={$x$-coordinates of the bottom line  $\Gamma_{bot}^f$},
    legend style={at={(0.03,0.95)},anchor=north west}
   ]
\addplot[blue, solid] table[x index=0, y index=1]{U3X.txt};
\addlegendentry{$k=1, h = 0.05$}
\addplot[red, densely dotted] table[x index=0, y index=2]{U3X.txt};
\addlegendentry{$k=2, h = 0.1$}
\end{axis}
\end{tikzpicture}

\begin{tikzpicture} 
     \tikzstyle{every node}=[font=\footnotesize]
\begin{axis}[
	width=\textwidth,
	height=0.26\textheight,
            xtick={0,1,2,3,4,5},
    ytick={-3000, 0,3000},
    xmin=-0.01,
    xmax=5.01,
    ymin=-3500,
    ymax=3500,
   	yticklabel style={/pgf/number format/fixed,/pgf/number format/precision=1},   	
   	every axis plot/.append style={line width=0.8pt, smooth},
    legend style={at={(0.03,0.95)},anchor=north west},
    ylabel={Pressure},
    xlabel={$x$-coordinates of the bottom line $\Gamma_{bot}^f$},
   ]
\addplot[blue, solid] table[x index=0, y index=1]{P3X.txt};
\addlegendentry{$k=1, h = 0.05$}
\addplot[red, densely dotted] table[x index=0, y index=2]{P3X.txt};
\addlegendentry{$k=2, h = 0.1$}
\end{axis}
\end{tikzpicture}

\begin{tikzpicture} 
     \tikzstyle{every node}=[font=\footnotesize]
\begin{axis}[
	width=\textwidth,
	height=0.26\textheight,
            xtick={0,2,4,6},
    ytick={-0.008,0,0.008},
    xmin=-0.01,
    xmax=5.01,
    ymin=-.008,
    ymax=0.008,
   	yticklabel style={/pgf/number format/fixed,/pgf/number format/precision=1},   	
   	every axis plot/.append style={line width=0.8pt, smooth},
    ylabel={Vertical Displacement},
    xlabel={$x$-coordinates of the interface $\Gamma$},
    legend style={at={(0.03,0.95)},anchor=north west}
   ]
\addplot[blue, solid] table[x index=0, y index=1]{D3X.txt};
\addlegendentry{$k=1, h = 0.05$}
\addplot[red, densely dotted] table[x index=0, y index=2]{D3X.txt};
\addlegendentry{$k=2, h = 0.1$}
\end{axis}
\end{tikzpicture}
\caption{\it \textbf{Example 3:}
  Numerical solutions of the scheme \eqref{full} with different
  discretization parameters 
  along cut lines
  at final time $t=1.2\times 10^{-2}[\mathsf{s}]$.
Top: flow rate $\frac23 \bld v_h[0]$ along 
center line $\{(x,0,0):\,x\in(0,5)\}$; 
Middle: pressure along center line 
$\{(x,0,0):\,x\in(0,5)\}$; 
Bottom: y-component of displacement 
$\bld\eta^s_h[1]$ along the interface line
$\{(x,0.5,0):\,x\in(0,5)\}$.}
  \label{fig:ex3Y}
\end{figure}
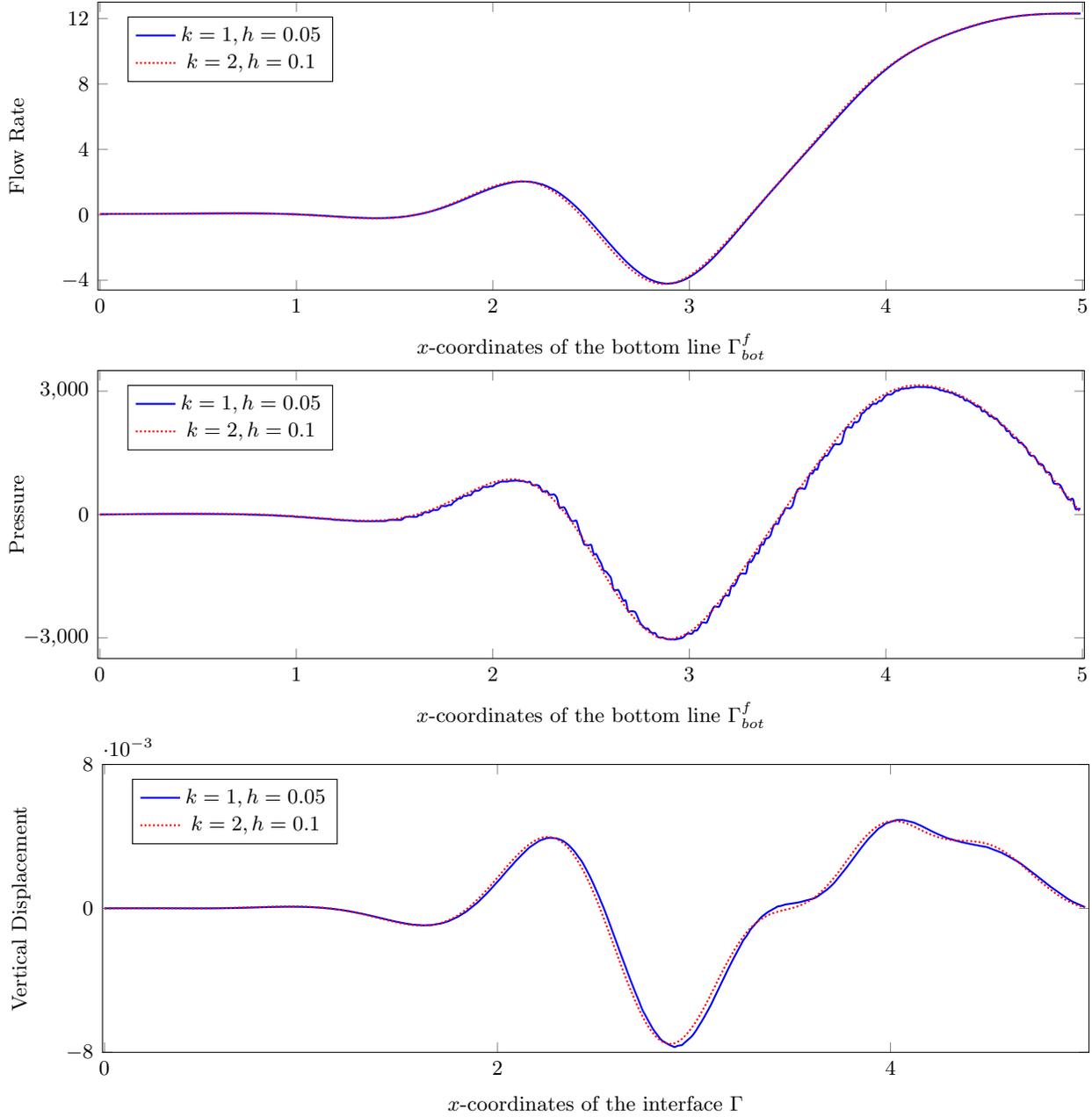

Finally, we plot 
the structure deformation along with the fluid pressure 
for $k=2$ with $h=0.1$ in Figure \ref{fig:ex3X} for
$t\in\{4,8,12\}\times 10^{-3}$. 
We clearly observe the propagation of a pressure pulse as time evolves.
\begin{figure}
  \centering
  \includegraphics[width=0.95\textwidth]{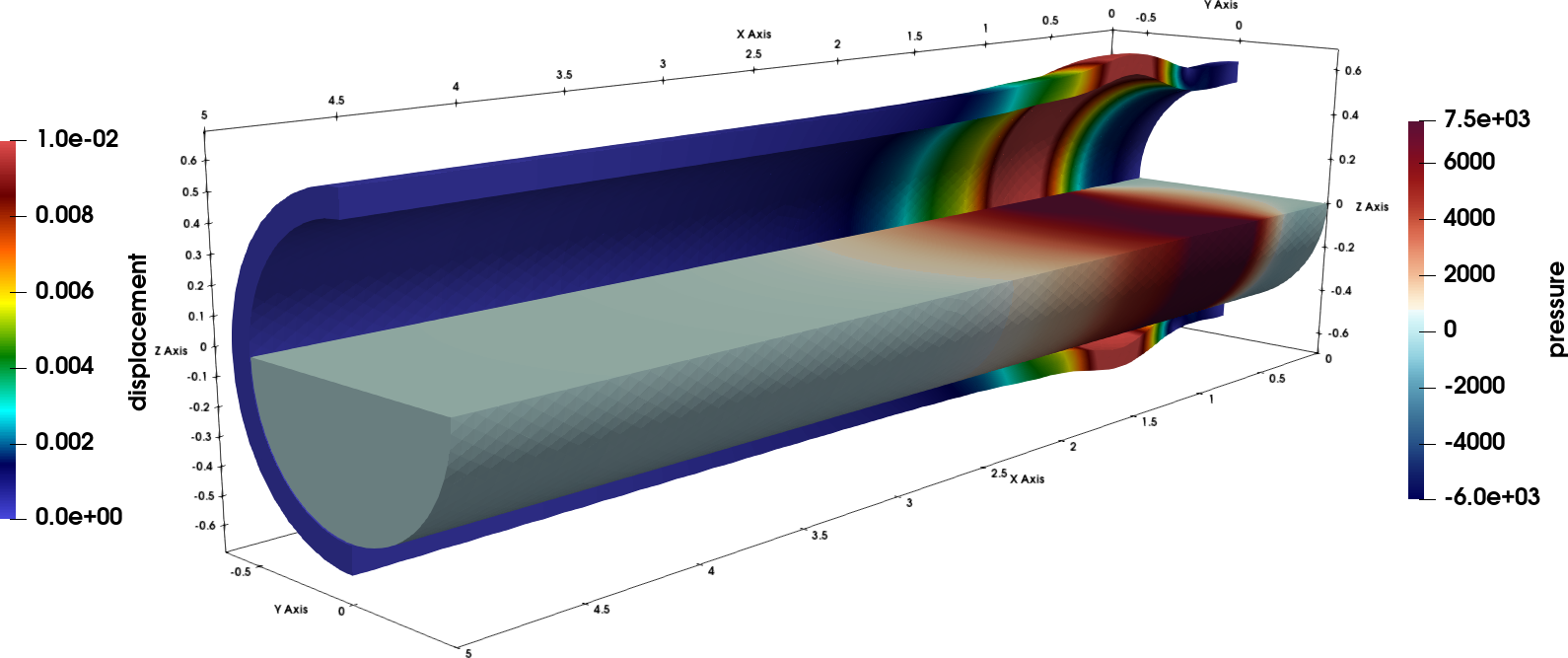}
  \includegraphics[width=0.95\textwidth]{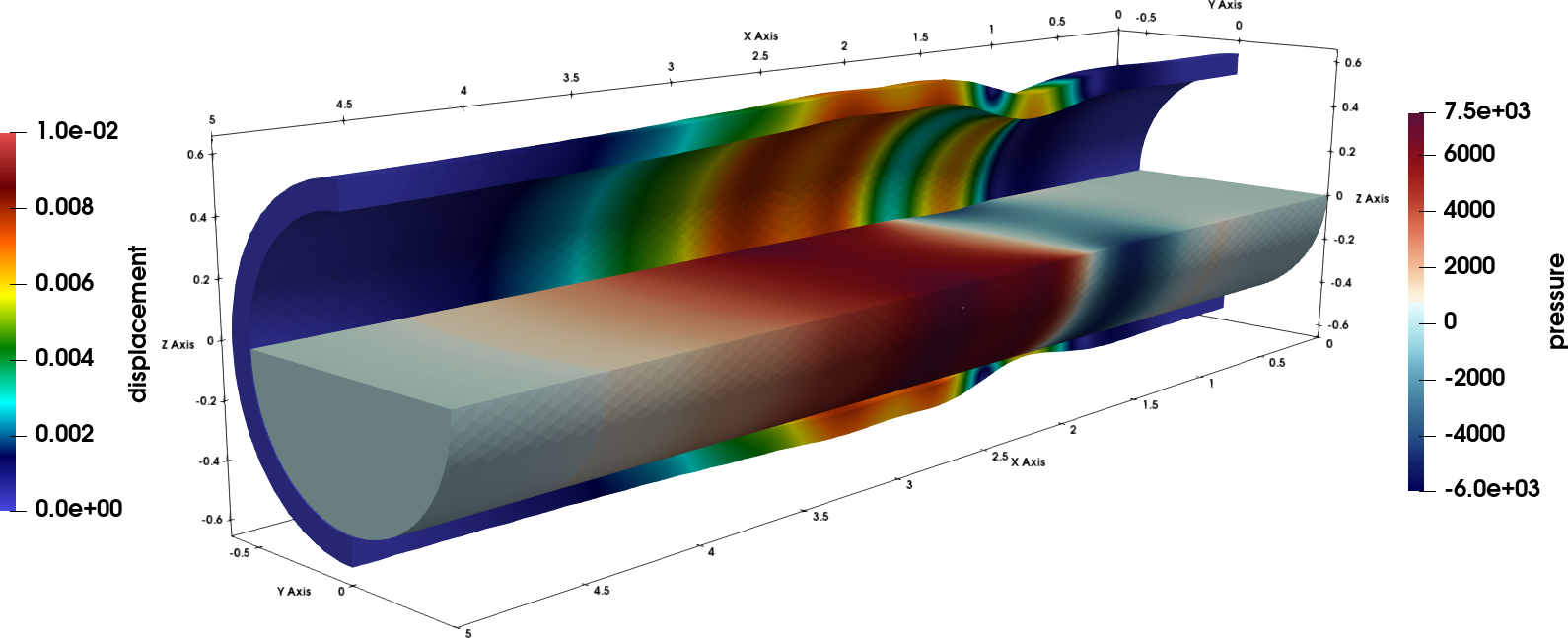}
  \includegraphics[width=0.95\textwidth]{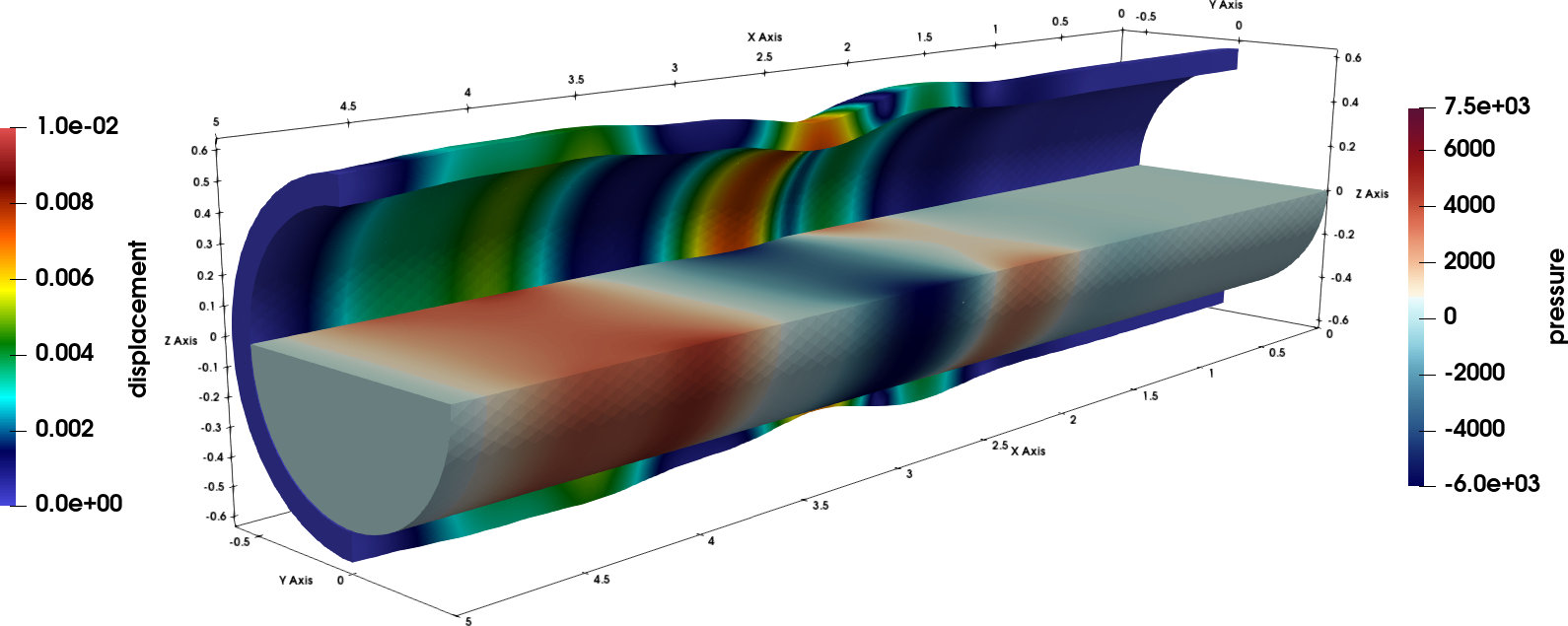}
  \caption{\it \textbf{Example 3:} The structure deformation and pressure 
approximation at different time.
 The structure deformation is enlarged by a factor of 8 and 
 is only shown on half of the structure domain with $y<0$. 
 The pressure approximation is only shown on half of the fluid domain with
 $z<0$. 
 From top to bottom, $t=0.004,
 0.008, 0.012$.
 ( $k=2$, $h=0.1$).
}
\label{fig:ex3X}
\end{figure}

\section{Conclusion}
\label{sec:conclude}
We have present a novel monolithic divergence-conforming HDG scheme for 
a linear FSI problem with a thick structure. The fully discrete scheme 
produces an exactly divergence-free fluid velocity approximation and 
{is energy-stable}.
Furthermore, we design an efficient block AMG
preconditioner and use it with a preconditioned MinRes solver for the resulting symmetric and
indefinite global linear system. This preconditioner is numerically observed
to be robust with respect to the mesh size, time step size and material
parameters in large parameter ranges. A theoretical analysis of this
preconditioner is our future work.

The extension of our scheme to other FSI models including thin structure
and/or moving interfaces consists of our ongoing work.

\label{sec:ale}
\bibliographystyle{siam}

\end{document}